\documentclass{amsart}
\usepackage{latexsym}
\usepackage{amsxtra}
\usepackage{amssymb}
\usepackage{amsthm}
\usepackage{xspace}

\usepackage{euscript}
\input xy
\xyoption{all} \CompileMatrices

\newtheorem{theorem}{Theorem}[section]
\newtheorem{proposition}[theorem]{Proposition}
\newtheorem{corollary}[theorem]{Corollary}
\newtheorem{lemma}[theorem]{Lemma}

\theoremstyle{definition}
\newtheorem{example}[theorem]{Example}

\newtheorem{remark}[theorem]{Remark}

\begin{document}

\def\SStar{{\operatorname{SStar}}}
\def\Star{{\operatorname{Star}}}
\def\QStar{{\operatorname{QStar}}}
\def\UU{{\mathcal U}}
\def\setmin{{-}}
\def\UU{{\mathcal U}}
\def\Z{{\mathcal Z}}
\def\PPP{{\mathcal P}}
\def\CCC{{\mathcal C}}
\def\XXX{{\mathcal X}}
\def\YYY{{\mathcal Y}}
\def\id{\operatorname{id}}
\def\Lin{\operatorname{Lin}}
\def\Spec{\operatorname{Spec}}
\def\Cl{\operatorname{Cl}}
\def\Idem{\operatorname{Idem}}
\def\Fix{\operatorname{Fix}}
\def\Cld{\operatorname{-Cl}}
\def\Inv{{\operatorname{Inv}}}
\def\Int{\operatorname{Int}}
\def\Pic{\operatorname{Pic}}
\def\SS{\mathcal{SS}}
\def\U{{\textsf{\textup{U}}}}
\def\V{{\textsf{\textup{V}}}}
\def\Inv{\operatorname{Inv}}
\def\Mor{\operatorname{Mor}}
\def\Int{{\textsf{\textup{Int}}}}
\def\I{{\textsf{\textup{I}}}}
\def\G{{\textsf{\textup{G}}}}
\def\K{{\textsf{\textup{K}}}}
\def\T{{\textsf{\textup{T}}}}
\def\R{{\textsf{\textup{R}}}}
\def\M{{\textsf{\textup{M}}}}
\def\FS{{\textsf{\textup{F}}}}
\def\Idl{{\textsf{\textup{Idl}}}}
\def\A{{\textsf{\textup{A}}}}
\def\C{{\textsf{\textup{C}}}}
\def\N{{\textsf{\textup{N}}}}
\def\Fp{{\mathcal F}}
\def\F{{\mathcal K}}
\def\S{\mathcal S}
\def\P{\operatorname{Prin}}
\def\B{\mathcal B}
\def\f{f}
\def\Inver{\operatorname{-Inv}}
\def\Prin{\textup{-}\mathcal{P}}
\def\Hom{\operatorname{Hom}}
\def\TM{t\textup{-Max}}
\def\TS{t\textup{-Spec}}
\def\Max{\operatorname{Max}}
\def\SSpec{\textup{-Spec}}
\def\SGV{\star\textup{-GV}}
\def\core{\operatorname{-core}}
\def\WBF{\operatorname{WBF}}
\def\sKr{\operatorname{sKr}}
\def\Loc{\operatorname{Loc}}
\def\Ass{\operatorname{Ass}}
\def\wAss{\operatorname{wAss}}
\def\ann{\operatorname{ann}}
\def\ind{\operatorname{ind}}
\def\cl{\star}
\def\ZZ{{\mathbb Z}}
\def\CC{{\mathbb C}}
\def\NN{{\mathbb N}}
\def\RR{{\mathbb R}}
\def\QQ{{\mathbb Q}}
\def\FF{{\mathbb F}}
\def\mm{{\mathfrak m}}
\def\nn{{\mathfrak n}}
\def\aaa{{\mathfrak a}}
\def\bbb{{\mathfrak b}}
\def\ppp{{\mathfrak p}}
\def\qqq{{\mathfrak q}}
\def\MM{{\mathfrak M}}
\def\qq{{\mathfrak Q}}
\def\rr{{\mathfrak R}}
\def\DD{{\mathfrak D}}
\def\cc{{\mathfrak S}}

\newcommand\bigsqcapp{\mathop{\mathchoice
	{\vcenter{\hbox{\huge $ \sqcap $}}}	
	{\vcenter{\hbox{\LARGE $ \sqcap $}}}	
	{\vcenter{\hbox{$ \sqcap $}}}		
	{\vcenter{\hbox{\small $ \sqcap $}}}}}	

\newcommand\bigsqcupp{\mathop{\mathchoice
	{\vcenter{\hbox{\huge$ \sqcup $}}}	
	{\vcenter{\hbox{\LARGE$ \sqcup $}}}	
	{\vcenter{\hbox{$ \sqcup $}}}		
	{\vcenter{\hbox{\small $ \sqcup $}}}}}	

\title[Prequantales]{Nuclei and applications to star, semistar, and semiprime operations}
\author{Jesse Elliott} \address{California
State University, Channel Islands\\ One University Drive \\ Camarillo, California 93012}
\email{jesse.elliott@csuci.edu}

\begin{abstract}
We show that the theory of quantales and quantic nuclei motivate new results on star operations, semistar operations, semiprime operations, ideal systems, and module systems, and conversely the latter theories motivate new results on quantales and quantic nuclei.  Results include representation theorems for precoherent prequantales and multiplicative semilattices; characterizations of the simple prequantales; and a generalization to the setting of precoherent quantales of the construction of the largest finite type semistar operation and the largest stable semistar operation smaller than a given semistar operation.
\end{abstract}

\maketitle



\section{Introduction}

Certain classes of ordered algebraic structures---Boolean algebras, Heyting algebras, multiplicative lattices, residuated lattices, locales, and quantales, to name a few---serve to unify aspects of logic, order theory, algebra, and topology.  One such set of unifications occured in the study, known as abstract ideal theory, of the ideal lattice of a commutative ring as a multiplicative lattice \cite{kru1}, initiated by Krull in the mid 1920s.  Another set of notions, including Brouwer lattices, Heyting algebras, and frames, led to the study of locales, or ``point-free topological spaces'' \cite{isb}.  These in turn led to a common generalization of multiplicative lattices and locales known as quantales, or ``quantum locales'', introduced in \cite{mul} by Mulvey to provide a lattice theoretic setting for the foundations of quantum mechanics and the theory of $\operatorname{C}^*$-algebras and to develop the concept of noncommutative topology introduced in \cite{gil} by Giles and Kummer.

This paper concerns applications of the theory of quantales to abstract ideal theory and commutative ring theory.   We show that the theory of nuclei on ordered magmas \cite[Section 3.4.11]{gal}, and more specifically on suitable generalizations of quantales, provides a noncommutative and nonassociative abstract ideal theoretic setting for the theories of star operations, semistar operations, ideal systems, and module systems, and conversely the latter theories motivate new results on quantales.  In this introduction we provide some background on the topics mentioned above and then provide a summary, addressing its organization.   In Sections \ref{sec:POM} through \ref{sec:stable} we provide new results on nuclei, as well as on various classes of ordered magmas, including quantales, multiplicative lattices, and multiplicative semilattices.  Then in Sections \ref{sec:ASO} through \ref{sec:AMSIS} we apply these results to star and semistar operations and ideal and module systems.  

A {\it magma} is a set $M$ equipped with a binary operation on $M$, which we write multiplicatively.  A magma is {\it unital} if it has an identity element.  An {\it ordered magma} is a magma $M$ equipped with a partial ordering $\leq$ on $M$ such that $x \leq x'$ and $y \leq y'$ implies $xy \leq x'y'$ for all $x,x',y,y' \in M$. 

For any self-map $\star$ of a set $X$ we write $x^\star = \star(x)$ for all $x \in X$ and $Y^\star = \star(Y)$ for all $Y \subseteq X$.    If $\star$ is a self-map of a magma $M$, then the binary operation $(x,y) \longmapsto x \star y = (xy)^\star$ on $M$ will be called {\it $\star$-multiplication}.   We then consider the set $M^\star$ as a magma under  $\star$-multiplication restricted to $M^\star$.  

A {\it closure operation} on a poset $S$ is a self-map $\star$ of $S$ satisfying the following conditions.
\begin{enumerate}
\item $\star$ is expansive: $x \leq x^\cl$ for all $x \in S$.
\item $\star$ is order-preserving: $x \leq y$ implies $x^\cl \leq y^\cl$ for all $x,y \in S$.
\item $\star$ is idempotent: $(x^\cl)^\cl = x^\cl$ for all $x \in S$.
\end{enumerate}
A closure operation on $S$ is equivalently a self-map $\star$ of $S$ such that $x \leq y^\cl$ if and only if $x^\cl \leq y^\cl$ for all $x, y \in S$.  A {\it nucleus} (resp., {\it strict nucleus}) on an ordered magma $M$ is a closure operation $\star$ on $M$ such that $x^\star y^\star \leq (xy)^\star$ (resp., $x^\star y^\star = (xy)^\star$) for all $x,y \in M$.   Nuclei were first studied in the contexts of ideal lattices and locales and later in the context of quantales as {\it quantic nuclei} \cite{nie}. 

\begin{example}\label{nearexamples}  \
\begin{enumerate}
\item For any commutative ring $R$, a {\it semiprime operation} on $R$ \cite{pet} is equivalently a nucleus on  the ordered magma $\mathcal{I}(R)$ of all ideals of $R$.  For example, the ideal radical operation $I \longmapsto \sqrt{I}$ is a semiprime operation on $R$, and if $R$ is Noetherian of prime characteristic then the operation of tight closure \cite{hun} is a semiprime operation on $R$.
\item By \cite[Lemmas 10.1--10.4]{giv}, for any topological space $X$ the operation $\operatorname{reg}: U \longmapsto \operatorname{int}(\operatorname{cl}(U))$ is a strict nucleus on the locale $\mathcal{O}(X)$ of open subsets of $X$, where $\operatorname{cl}$ and $\operatorname{int}$ denote the topological closure and interior operations, respectively.  (Open sets $U$ such that $U = U^{\operatorname{reg}}$ are called {\em regular open sets}.)
\end{enumerate}
\end{example}

Star, semistar, and semiprime operations are the most prominent examples of nuclei airising in commutative algebra.  They are useful, among other reasons, for generalizing results on various classes of rings to much larger classes by relaxing, up to closure, certain ideal-theoretic assumptions, and also for specializing to classes of rings for which all or certain ideals are closed with respect to a given closure operation.  Examples  are provided by the following equivalences.
\begin{enumerate}
\item A domain is a UFD iff every ideal is principal up to $t$-closure. 
\item A domain is Dedekind iff every nonzero ideal is invertible, while a domain is Krull iff every nonzero ideal is invertible up to $t$-closure. 
\item A domain is  Pr\"ufer iff every nonzero ideal is $t$-closed.
\item  A domain is divisorial iff every nonzero ideal is closed under the  divisorial closure operation.
\item A ring is reduced iff every ideal is closed under the radical operation.
\item  A Noetherian ring of prime characteristic is weakly F-regular iff every ideal is tightly closed, and is  F-rational iff all parameter ideals are tightly closed. 
\end{enumerate}

The theory of nuclei allows for a unified treatment of closure operations as they naturally arise, not only in commutative algebra, but also in logic, order theory, the theory of quantales, and the theory of $\mbox{C}^*$-algebras.  Further motivation for studying nuclei is provided by the fact that, not only semiprime operations, but also star and semistar operations, can be characterized as nuclei.  Let $D$ be an integral domain with quotient field $F$, and let $\F(D)$ denote the ordered monoid of all nonzero $D$-submodules of $F$ under the operation of multiplication.  A {\it semistar operation} on $D$ is a closure operation $\star$ on the poset $\F(D)$ such that $(aI)^\star = aI^\star$ for all nonzero $a \in F$ and all $I \in \F(D)$.  Semistar operations were introduced in \cite{oka} as a generalization of {\it star operations}, which were introduced by Krull in \cite[Section 6.43]{kru}.  The following result is proved in Section \ref{sec:ASO}.

\begin{theorem}\label{mainprop1}
Let $D$ be an integral domain with quotient field $F$.  The following conditions are equivalent for any self-map $\star$ of $\F(D)$.
\begin{enumerate}
\item $\star$ is a semistar operation on $D$.
\item $\star$ is a nucleus on the ordered monoid $\F(D)$.
\item $\star$ is a closure operation on the poset $\F(D)$ and $\star$-multiplication on $\F(D)$ is associative.
\item $\star$ is a closure operation on the poset $\F(D)$ such that $(I^\star J^\star)^\star = (IJ)^\star$ for all $I,J \in \F(D)$ (or equivalently such that the map $\star: \F(D) \longrightarrow \F(D)^\star$ is a magma homomorphism).
\item $HJ \subseteq I^\star$ if and only if $HJ^\star \subseteq I^\star$ for all $H,I,J \in \F(D)$.
\item $(I^\star :_F J) = (I^\star :_F J^\star)$ for all $I, J \in \F(D)$.
\end{enumerate}
\end{theorem}

\begin{corollary}\label{latticeisom}
Let $D$ and $D'$ be integral domains.  If the ordered monoids $\F(D)$ and $\F(D')$ are isomorphic, then the lattices of all semistar operations on $D$ and $D'$, respectively, are isomorphic.
\end{corollary}

A $D$-submodule $I$ of $F$ is said to be a {\it fractional ideal} of $D$ if $aI \subseteq D$ for some nonzero element $a$ of $F$.  Let $\Fp(D)$ denote the set of all nonzero fractional ideals of $D$, which is an ordered submonoid of $\F(D)$.  A {\it star operation} $*$ on $D$ is a closure operation on the poset $\Fp(D)$ such that $D^* = D$ and $(aI)^* = aI^*$ for all nonzero $a \in F$ and all $I \in \Fp(D)$.  Theorem \ref{mainprop2} provides a characterization of star operations analogous to that of Theorem \ref{mainprop1} above.  In particular, a star operation on $D$ is equivalently a nucleus $*$ on the ordered monoid $\Fp(D)$ such that $D^* = D$.  The weak ideal systems and module systems of \cite{hal1, hal}, which generalize semiprime operations and star and semistar operations, respectively, can also be characterized as nuclei: see Section \ref{sec:AMSIS}.

By Theorem \ref{mainprop1}, one can characterize a semistar operation on $D$ as a self-map $\star$ of $\F(D)$ satisfying a single axiom, namely, statement (5) (or (6)) of the theorem.   The theorem suggests that there is potential overlap among the theory of semistar operations, abstract ideal theory, and the theory of quantales and quantic nuclei.  The connections among these theories suggested by Theorem \ref{mainprop1} and its equivalents for star operations and semiprime operations are the main inspiration for this paper.

Quantales provide a natural context in which to study nuclei, as the nuclei on a quantale classify its quotient objects in the category of quantales.  A {\it quantale} (resp., {\it prequantale}) is an ordered semigroup (resp., ordered magma) $Q$ such that the supremum $\bigvee X$ of $X$ exists and $a (\bigvee X) = \bigvee(aX)$ and $(\bigvee X)a = \bigvee(Xa)$ for any $a \in Q$ and any subset $X$ of $Q$ \cite[Definition 2.4.2]{ros}.  A quantale is equivalently an associative prequantale.  A {\it multiplicative lattice} is a commutative unital quantale,  and a  {\it locale}, or {\it frame}, is a quantale in which multiplication is the operation $\wedge$ of infimum.  A prequantale (resp., quantale, multiplicative lattice) is equivalently a magma object (resp., semigroup object, commutative monoid object) in the monoidal category of sup-lattices, that is, the category of complete lattices equipped with the tensor product with morphisms as the sup-preserving functions \cite[Section II.5]{joy}. 

\begin{example} \
\begin{enumerate}
\item The power set $2^M$ of a magma $M$ is a prequantale under the operation $(X,Y) \longmapsto XY = \{xy: \ x \in X, y \in Y\}$.  It is the free prequantale on the magma $M$. The prequantale $2^M$ is a quantale if and only if $M$ is a semigroup, so $2^M$ is the free quantale on $M$ if $M$ is a semigroup.
\item The poset of all supremum-preserving self-maps of a complete lattice is a unital quantale under the operation of composition.
\item Any totally ordered complete lattice is a frame, as is any complete Boolean algebra.
\item The complete lattice $\mathcal{O}(X)$ of all open subsets of a topological space $X$ is a locale.
\item For any algebra $A$ over a ring $R$, the complete lattice $\operatorname{Mod}_R(A)$ of all (two-sided) $R$-submodules of $A$ is a quantale under the operation $(I,J) \longmapsto IJ = \{\sum_{x \in X, y \in Y} xy: X \subseteq I, Y \subseteq J \mbox{ finite}\}$.
\item The poset  $\operatorname{Max}(A)$ of all closed linear subspaces of a unital $\mbox{C}^*$-algebra $A$ is a unital quantale with involution \cite[Definition 2.4]{krum2} under the operation $(M,N) \longmapsto \overline{MN}$ and by \cite[Theorem 5.4]{krum2} is a complete invariant of $A$.
\end{enumerate}
\end{example} 

In the definition above of quantales and prequantales, restricting subsets $X$ to be among certain subclasses of subsets of $Q$ yields more general structures with respect to which it is natural to study nuclei.  This is required, for example, to accommodate the theory of star operations, since the ordered monoids on which star operations are defined are bounded complete but not complete.   We will say that a {\it near prequantale} (resp., {\it semiprequantale}) is an ordered magma $Q$ such that $\bigvee X$ exists and $a (\bigvee X) = \bigvee(aX)$ and $(\bigvee X)a = \bigvee(Xa)$ for any $a \in Q$ and any nonempty subset $X$ (resp., any nonempty finite or bounded subset $X$) of $Q$.   A near prequantale is equivalently a semiprequantale $Q$ with a largest element, and a prequantale is equivalently a near prequantale $Q$ with a smallest element $\bigwedge Q$ that annihilates every element of $Q$. We define the terms {\it near quantale}, {\it semiquantale}, {\it near multiplicative lattice}, and {\it semimultiplicative lattice} in the obvious way.  A theme of this paper is that many known results about quantales and multiplicative lattices generalize to these more general structures.  In particular, the attempt is made to state each result in as general a setting as possible.  For example, results on star operations are generalized to the context of nuclei on precoherent semimultiplicative lattices.

\begin{example} \
\begin{enumerate}
\item If $a$ is an idempotent element of a prequantale $Q$, then the set $Q_{\leq a} = \{x \in Q: x \leq a \}$ is a subprequantale of $Q$, while the set $Q_{\geq a} = \{x \in Q : x \geq a\}$ is a sub near prequantale of $Q$.
\item If $Q$ is a prequantale and $x y = 0$ implies $x = 0$ or $y = 0$ for all $x,y \in Q$, where $0 = \bigwedge Q$,  then $Q\setmin\{0\}$ is a sub near prequantale of $Q$.  In particular, the set $2^M \setmin \{\emptyset\}$ of all nonempty subsets of a magma $M$ is a sub near prequantale of the prequantale $2^M$, and it is the free near prequantale on the magma $M$. 
\end{enumerate}
\end{example}

The remainder of this paper is organized as follows.  In Section \ref{sec:POM} we introduce several classes of ordered magmas.  In Sections \ref{sec:MCO} and \ref{sec:COCL} we study properties of nuclei and the poset of all nuclei on an ordered magma, with particular attention to the classes of ordered magmas defined in Section \ref{sec:POM}.    In Section \ref{sec:DC} we generalize the generalized divisorial closure semistar operations \cite[Example 1.8(2)]{pic} to the contexts of near prequantales and residuated ordered monoids, and we use this to give a characterization of the simple near prequantales. 

A closure operation $\star$ on a poset $S$ is said to be {\it finitary} if $(\bigvee \Delta)^\star = \bigvee (\Delta^\star)$ for all directed subsets $\Delta$ of $S$ for which $\bigvee \Delta$ exists.  For example, a  {\it finite type} semistar operation \cite[Definition 3]{oka} on an integral domain $D$ is equivalently a finitary nucleus on $\F(D)$.  In Section \ref{sec:PFN} we examine the poset of all finitary nuclei on an ordered magma.   We also generalize the well-known construction of the largest finite type semistar operation $\star_f$ smaller than a semistar operation $\star$, as follows.  An element $x$ of a poset $S$ is said to be {\it compact} if whenever $x \leq \bigvee \Delta$ for some directed subset $\Delta$ of $S$ for which $\bigvee \Delta$ exists one has $x \leq y$ for some $y \in \Delta$.  Generalizing the notion of a precoherent quantale \cite[Definition 4.1.1(2)]{ros}, we say that an ordered magma $M$ is {\it precoherent} if every element of $M$ is the supremum of a set of compact elements of $M$ and the set $\K(M)$ of all compact elements of $M$ is a submagma of $M$.  For example, the near multiplicative lattice $\F(D)$ is precoherent for any integral domain $D$ with quotient field $F$ since $\K(\F(D))$ is the set of all nonzero finitely generated $D$-submodules of $F$. The following result is proved in Section \ref{sec:PFN}.

\begin{theorem}\label{precoherentthm}
If $\star$ is any nucleus on a precoherent semiprequantale $Q$, then there is a largest finitary nucleus $\star_f$ on $Q$ that is smaller than $\star$, and one has $x^{\star_f} = \bigvee\{y^\star: \ y \in \K(Q) \mbox{ and } \ y \leq x\}$ for all $x \in Q$.
\end{theorem}

Also in Section \ref{sec:PFN}, we prove that the association $Q \longmapsto \K(Q)$ yields an equivalence between the category of precoherent near prequantales (resp., precoherent prequantales) and the category of multiplicative semilattices (resp., prequantic semilattices).  This result, along with Theorem \ref{precoherentthm} above, provides substantive motivation for the study of precoherent near prequantales, particularly in algebraic settings.

In Section \ref{sec:stable} we generalize several known results and new results on stable semistar operations to the context of precoherent semimultiplicative lattices, where a semistar operation $\star$ on an integral domain $D$ is said to be {\it stable} if  $(I \cap J)^\star = I^\star \cap J^\star$ for all $I,J \in \F(D)$ \cite{fon0}. 

Finally, in Sections \ref{sec:ASO} through \ref{sec:AMSIS} we apply the theory of nuclei to  star and semistar operations and ideal and module systems.  There, for example, we prove Theorem \ref{mainprop1} and its equivalents for star operations and ideal and module systems, we give a construction of the smallest semistar operation extending a given star operation, and we provide some new examples of semistar operations induced by complete integral closure, plus closure, and tight closure.  The last of these examples leads to two potentially useful definitions of tight closure as a semiprime operation on any commutative ring, not necessarily Noetherian, of prime characteristic.  

We remark that this paper is a substantial revision of a previous version \cite{ell2} posted on arxiv.org in 2011 under a different title.




\section{Ordered algebraic structures}\label{sec:POM}

If a proof in this paper is omitted then its reconstruction should be routine.  All rings and algebras are assumed unital and all magmas are written multiplicatively.  For any subset $X$ of a poset $S$ we write $\bigvee X = \bigvee_S X$ for the supremum of $X$ and $\bigwedge X = \bigwedge_S X$ for the infimum of $X$ if either exists.  (Note the cases $\bigvee \emptyset = \bigwedge S$ and $\bigwedge \emptyset = \bigvee S$.)

The {\it category of posets} has as morphisms the order-preserving functions.
A poset in which every pair of elements of has a supremum (resp., infimum) is said to be a {\it join semilattice} (resp., {\it meet semilattice}).  A {\it lattice} is a poset that is both a join and meet semilattice.  A poset $S$ is {\it complete} if every subset of $S$ has a supremum in $S$, or equivalently if every subset of $S$ has an infimum $S$.  A complete poset is also known as a {\it complete lattice}, or {\it sup-lattice}.  A poset $S$ is {\it bounded complete} if every nonempty subset of $S$ that is bounded above has a supremum in $S$, or equivalently if every nonempty subset of $S$ that is bounded below has an infimum in $S$. 

We will say that a poset $S$ is {\it near sup-complete}, or a {\it near sup-lattice}, if every nonempty subset of $S$ has a supremum in $S$. A near sup-lattice is equivalently a bounded complete poset with a largest element, or equivalently a poset $S$ such that every subset of $S$ that is bounded below has an infimum in $S$.   A sup-lattice is equivalently a near sup-lattice having a least element.  A map $f: S \longrightarrow T$ between posets is said to be {\it sup-preserving} if $f(\bigvee X) = \bigvee f(X)$ for every subset $X$ of $S$ for which $\bigvee X$ exists.   We will say that $f: S \longrightarrow T$ is {\it near sup-preserving} if $f(\bigvee X) = \bigvee f(X)$ for every nonempty subset $X$ of $S$ such that $\bigvee X$ exists.  If $f$ is near sup-preserving, then $f$ is sup-preserving if and only if $f(\bigwedge S) = \bigwedge T$ or $\bigwedge S$ does not exist.  The morphisms in the {\it category of sup-lattices} (resp., {\it category of near sup-lattices}) are the sup-preserving (resp., near sup-preserving) functions. 

A nonempty subset $\Delta$ of a poset $S$ is said to be {\it directed} if every finite subset of $\Delta$ has an upper bound in $\Delta$.   A poset $S$ is {\it directed complete}, or a {\it dcpo}, if each of its directed subsets has a supremum in $S$.  We will say that $S$ is a {\it bdcpo} if every directed subset of $S$ that is bounded above has a supremum in $S$.  A subset $X$ of $S$ is said to be {\it downward closed} (resp., {\it upward closed}) if $y \in X$ whenever $y \leq x$ (resp., $y \geq x$) for some $x \in X$.  The set $X$ is said to be {\it Scott closed} if $X$ is a downward closed subset of $S$ and for any directed subset $\Delta$ of $X$ one has $\bigvee \Delta \in X$ if $\bigvee \Delta$ exists.  The set $X$ is {\it Scott open} if its complement is Scott closed, or equivalently if $X$ is an upward closed subset of $S$ and $X \cap \Delta \neq \emptyset$ for any directed subset $\Delta$ of $S$ with $\bigvee \Delta \in X$.  The Scott open subsets of $S$ form a topology on $S$ called the {\it Scott topology}.  A function $f: S \longrightarrow T$ between posets is said to be {\it Scott continuous} if $f$ is continuous when $S$ and $T$ are endowed with the Scott topologies.  Equivalently, $f$ is Scott continuous if and only if $f(\bigvee \Delta) = \bigvee f(\Delta)$ for every directed subset $\Delta$ of $S$ for which $\bigvee \Delta$ exists.  Every near sup-preserving function is Scott continuous, and every Scott continuous function is order-preserving.

If $S$ and $T$ are posets, then $S \times T$ is a poset under the relation $\leq$ defined by $(x,y) \leq (x',y')$ iff $x \leq x'$ and $y \leq y'$. The Scott topology on $S \times T$ may be strictly finer than the product of the Scott topologies on $S$ and $T$ \cite[Exercise II-4.26]{gie}.  

Characterizations of the classes of ordered magmas that will be discussed in this section are given in Table 1, and the various relationships among these classes are summarized in the lattice diagrams of Figures 1, 2, and 3.  Within any of these three diagrams the intersection of any two of the classes is the class lying directly above both of them. 

\begin{table} 
\caption{Characterizations of ordered magmas (om's)}
\centering  \ \\
\begin{tabular}{l|l|l|l} 
ordered magma $M$ & abbr. & $\bigvee X$ exists & $\bigvee(XY) = \bigvee X \bigvee Y$ \\
\hline\hline
sup-magma & s &	for all $X$ &   \\ \hline
near sup-magma & ns &	for all $X \neq \emptyset$ &   \\ \hline
dcpo magma & d & 	for all directed $X$ &   \\ \hline
bounded complete & bc & 	for all $X \neq \emptyset$  &   \\ 
 & &  bounded above	 &   \\ \hline
bounded above & b & 	for $X = M$ &   \\ \hline
with annihilator & a &	for $X = \emptyset$ &  for $X = \emptyset$ or $Y = \emptyset$ \\ \hline
prequantale & p &  for all $X$ & for all $X,Y$ \\ \hline
near prequantale & np &  for all $X \neq \emptyset$ & for all $X,Y \neq \emptyset$ \\ \hline
semiprequantale & sp &  for all finite or    &  for all (finite or)  \\
&  & bounded $X \neq \emptyset$ & bounded $X,Y \neq \emptyset$ \\ \hline
 prequantic  & ps & for all finite $X$ & for all finite $X,Y$ \\ 
 semilattice & & & \\ \hline 
multiplicative & ms &  for all finite $X \neq \emptyset$   &  for all finite $X,Y \neq \emptyset$ \\
semilattice &  &  & \\ \hline
Scott topological & t &			 & for all directed $X,Y$ \\
	 & 			& &  such that $\bigvee X$, $\bigvee Y$  \\ 
	 	 & 		&	 &  exist \\ \hline
residuated & r & if $\exists x,y : \ X = $ & for all $X,Y$ such that \\ 
	 & & $\{z: zy \leq x\}$ or &  $\bigvee X$, $\bigvee Y$ exist  \\ 
	 & & $X = \{z: yz \leq x\}$ & \\ \hline
near residuated & nr & if $X \neq \emptyset$ and $\exists x,y :$ & for all $X,Y \neq \emptyset$ such \\ 
	 & & $X = \{z: zy \leq x\}$ or &   that $\bigvee X$, $\bigvee Y$ exist  \\ 
	 & & $X = \{z: yz \leq x\}$ &  
\end{tabular}
\end{table}

\begin{figure}
\caption{Relationships among classes of ordered magmas}
\begin{eqnarray*}
\SelectTips{cm}{11}\xymatrix @R=.5cm @C=.6cm {
 && & {\mbox{p}}  \ar@{=>}[dl] \ar@{=>}[dr] & & & \\ 
 && {\mbox{r+ns}} \ar@{=>}[dl] \ar@{=>}[dr] & & {\mbox{np+s}} \ar@{=>}[dl] \ar@{=>}[dr] & & \\ 
 &{\mbox{r+d}} \ar@{=>}[dr] \ar@{=>}[dl] & & {\mbox{np}} \ar@{=>}[dl] \ar@{=>}[dr] & & {\mbox{t+s}}  \ar@{=>}[dr] \ar@{=>}[dl] & \\ 
 {\mbox{r}} \ar@{=>}[dr] & & {\mbox{nr+d}} \ar@{=>}[dr] \ar@{=>}[dl] & & {\mbox{t+ns}} \ar@{=>}[dl] \ar@{=>}[dr] & & {\mbox{s}} \ar@{=>}[dl] \\ 
 & {\mbox{nr}} \ar@{=>}[dr]& & {\mbox{t+d}} \ar@{=>}[dr] \ar@{=>}[dl] & &  {\mbox{ns}} \ar@{=>}[dl] & \\
 &  & {\mbox{t}} \ar@{=>}[dr] &  & {\mbox{d}} \ar@{=>}[dl] & &  \\
 &  &  & {\mbox{om}} & & & }
\end{eqnarray*}
\end{figure}

\begin{figure}
\caption{Further relationships among classes of ordered magmas}
\begin{eqnarray*}
\SelectTips{cm}{11}\xymatrix @R=.5cm @C=.6cm {
   & & {\mbox{p}} \ar@{=>}[dl] \ar@{=>}[dr] \ar@{=>}[d] &  \\ 
 & {\mbox{ps+t}} \ar@{=>}[dl] \ar@{=>}[d] \ar@{=>}[dr] & {\mbox{ps+s}}  \ar@{=>}[dr] \ar@{=>}[dl] & {\mbox{np}} \ar@{=>}[dl] \ar@{=>}[d]   \\ 
 {\mbox{t+a}}  \ar@{=>}[d]  \ar@{=>}[dr] & {\mbox{ps}} \ar@{=>}[dr] \ar@{=>}[dl] & {\mbox{ms+t}} \ar@{=>}[d] \ar@{=>}[dl] & {\mbox{ms+ns}}  \ar@{=>}[dl] \\ 
 {\mbox{a}}  \ar@{=>}[dr]  &  {\mbox{t}}  \ar@{=>}[d]  & {\mbox{ms}}  \ar@{=>}[dl]  &  \\
 & {\mbox{om}} & &  }
\end{eqnarray*}
\end{figure}


\begin{figure}
\caption{Further relationships among classes of ordered magmas}
\begin{eqnarray*}
\SelectTips{cm}{11}\xymatrix @R=.5cm @C=.6cm {
   & & {\mbox{p}} \ar@{=>}[dl] \ar@{=>}[dr] \ar@{=>}[d] &  & \\ 
 & {\mbox{ps+t+b}} \ar@{=>}[dl] \ar@{=>}[d] \ar@{=>}[dr] & {\mbox{ps+s}}  \ar@{=>}[dr] \ar@{=>}[dl] & {\mbox{np}} \ar@{=>}[dl] \ar@{=>}[d] \ar@{=>}[dr] &  \\ 
 {\mbox{t+a+b}} \ar@{=>}[d]  \ar@{=>}[dr] & {\mbox{ps+b}} \ar@{=>}[dr] \ar@{=>}[dl] & {\mbox{ms+t+b}} \ar@{=>}[d] \ar@{=>}[dl] \ar@{=>}[dr] & {\mbox{ms+ns}} \ar@{=>}[dr] \ar@{=>}[dl] &  {\mbox{sp}} \ar@{=>}[dl] \ar@{=>}[d] \\ 
 {\mbox{a+b}}   \ar@{=>}[dr] &  {\mbox{t+b}}  \ar@{=>}[dr] \ar@{=>}[d]  & {\mbox{ms+b}}  \ar@{=>}[dl]  \ar@{=>}[dr] & {\mbox{ms+t}}  \ar@{=>}[dl] \ar@{=>}[d] &  {\mbox{ms+bc}} \ar@{=>}[dl]  \\
 & {\mbox{b}} \ar@{=>}[dr]  & {\mbox{t}} \ar@{=>}[d]  & {\mbox{ms}} \ar@{=>}[dl] & \\
 &  & {\mbox{om}} & &  }
\end{eqnarray*}
\end{figure}

A magma is {\it unital} (resp., {\it left unital}, {\it right unital}) if it has an identity element (resp., left identity element, right identity element).  A map $f: M \longrightarrow N$ of magmas is a {\it homomorphism} of magmas if $f(xy) = f(x)f(y)$ for all $x,y \in M$.   
The morphisms in the {\it category of ordered magmas} are the order-preserving magma homomorphisms.  We will say that a {\it sup-magma} (resp., {\it near sup-magma}, {\it dcpo magma}, {\it bdcpo magma}) is an ordered magma that is complete (resp., near sup-complete, directed complete, a bdcpo) as a poset.  The morphisms in the {\it category of sup-magmas} (resp., {\it category of near sup-magmas}, {\it category of dcpo magmas},  {\it category of bdcpo magmas}) are the sup-preserving (resp., near sup-preserving, Scott continuous, Scott continuous) magma homomorphisms.

An {\it annihilator} of an ordered magma $M$ is a least element $0 = \bigwedge M$ of $M$ such that $0x = 0 = x0$ for all $x \in M$.  If an annihilator of $M$ exists then we say that $M$ is {\it with annihilator}.

\begin{lemma}\label{scott2}
The following are equivalent for any ordered magma $M$.
\begin{enumerate}
\item The map $M \times M  \longrightarrow M$ of multiplication in $M$ is sup-preserving and $\bigwedge M$ is an annihilator of $M$ if $\bigwedge M$ exists (resp., multiplication in $M$ is near sup-preserving, multiplication in $M$ is Scott continuous).
\item For all $a \in M$, the left and right multiplication by $a$ maps on $M$ are sup-preserving (resp., near sup-preserving, Scott continuous).
\item $a (\bigvee X) = \bigvee(a X)$ and $(\bigvee X) a = \bigvee(Xa)$ for any $a \in M$ and any subset (resp., any nonempty subset, any directed subset) $X$ of $M$ such that $\bigvee X$ exists.
\item $\bigvee (X Y) = \bigvee X \bigvee Y$ for any subset (resp., any nonempty subset, any directed subset) $X$ and $Y$ of $M$ such that $\bigvee X$ and $\bigvee Y$ exist.
\end{enumerate}
\end{lemma}

We will say that an ordered magma $M$ is {\it Scott-topological} if the multiplication map $M \times M \longrightarrow M$ is Scott continuous.  A Scott-topological ordered magma is equivalently a magma object in the monoidal category of posets equipped with the product $\times$ with morphisms as the Scott continuous functions.  
If $M$ is a topological magma when endowed with the Scott topology, then $M$ is a Scott-topological magma, but the converse appears to be false.

A near prequantale is equivalently a magma object in the monoidal category of near sup-lattices equipped with the product $\times$.  They form a full subcategory of the category of near sup-magmas.  A near quantale (resp., near multiplicative lattice) is equivalently a semigroup object (resp., commutative monoid object) in the monoidal category of near sup-lattices.  See Example \ref{nearexamples} for examples.  For a further example, note that the poset of all near sup-preserving self-maps of a near sup-lattice is a unital near quantale under the operation of composition.

An element $a$ of an ordered magma $M$ is said to be {\it residuated} if for all $x \in M$ there exists a largest element $z = x/a$ of $M$ such that $za \leq x$ and a largest element $z' = a \backslash x$ of $M$ such that $az' \leq x$.  (Many authors denote $x/a$ and $a \backslash x$ by $x \leftarrow a$ and $a \rightarrow x$, respectively.)  An ordered magma $M$ is said to be {\it residuated} if every element of $M$ is residuated \cite{ward}.  By \cite[Theorem 3.10]{gal}, if $M$ is residuated, then for all $a \in M$ the left and right multiplication by $a$ maps on $M$ are sup-preserving.

\begin{example}\label{nearresexamples} \
\begin{enumerate}
\item By \cite[Corollary 3.11]{gal} a prequantale is equivalently a complete and residuated ordered magma.
\item A {\it Heyting algebra} is a bounded lattice that, as an ordered monoid under the operation $\wedge$ of infimum, is residuated \cite[Section 1.1.4]{gal}.  A complete Heyting algebra, also known as a  {\it locale}, or {\it frame} (although the morphisms in their respective categories are diffierent), is equivalently a unital quantale (or prequantale) in which every element is idempotent and whose largest element is the identity element.
\item If $S$ is a nontrivial complete lattice, then $S$ is a near multiplicative lattice under the operation $\vee$ of supremum, but $S$ is not residuated since one has $x \vee \bigvee\{y \in S: \ x \vee y \leq 1\} = x > 1$ for all $x \neq 1$ in $S$.  
\end{enumerate} 
\end{example}

\begin{proposition}\label{quantales}
The following are equivalent for any ordered magma $M$.
\begin{enumerate}
\item $M$ is a prequantale.
\item $M$ is complete and residuated.
\item $M$ is a near prequantale with annihilator.
\item $M$ is complete with annihilator and the multiplication map $M \times M \longrightarrow M$ is sup-preserving.
\item $M$ is complete and the left and right multiplication by $a$ maps on $M$ are sup-preserving for all $a \in M$.
\item The map $2^M \longrightarrow M$ acting by $X \longmapsto \bigvee X$ is a well-defined magma homomorphism.
\item $\bigvee (X Y) = \bigvee X \bigvee Y$ for all subsets $X$ and $Y$ of $M$.
\end{enumerate}
\end{proposition}

We say that an element $a$ of an ordered magma $M$ is {\it near residuated} if for all $x \in M$ such that $z a \leq x$ (resp., $a z \leq x$) for some $z \in M$ there exists a largest such element $z = x/a$ (resp., $z = a \backslash x$) of $M$.  We say that $M$ {\it near residuated} if every element of $M$ is near residuated.  For example, Example \ref{nearresexamples}(3) is near residuated.  If $M$ is near residuated, then for all $a \in M$ the left and right multiplication by $a$ maps on $M$ are near sup-preserving.  Although near prequantales are not necessarily residuated, they are near residuated.  

\begin{proposition}\label{nearprequantales}
The following are equivalent for any ordered magma $M$.
\begin{enumerate}
\item $M$ is a near prequantale.
\item $M$ is near sup-complete and near residuated.
\item The ordered magma $M_0 = M \amalg \{0\}$, where $0x = 0 = x0$ and $0 \leq x$ for all $x \in M_0$, is a prequantale. 
\item $M$ is near sup-complete and the multiplication map $M \times M \longrightarrow M$ is near sup-preserving.
\item $M$ is near sup-complete and the left and right multiplication by $a$ maps on $M$ are near sup-preserving for all $a \in M$.
\item The map $2^M \setmin\{\emptyset\} \longrightarrow M$ acting by $X \longmapsto \bigvee X$ is a well-defined magma homomorphism.
\item $\bigvee (X Y) = \bigvee X \bigvee Y$ for all nonempty subsets $X$ and $Y$ of $M$.
\end{enumerate}
\end{proposition}

A {\it multiplicative semilattice} is an ordered magma $M$ such that $M$ is a join semilattice and $a (x \vee y) = ax \vee ay$ and $(x \vee y)a = xa \vee ya$ for all $a,x,y \in M$ \cite[Section XIV.4]{bir}.  We will say that a {\it prequantic semilattice} is a multiplicative semilattice with annihilator.    The prequantic semilattices (resp., multiplicative semilattices) form a category, where a morphism is a magma homomorphism $f: M \longrightarrow M'$ such that $f(\bigvee X) = \bigvee f(X)$ for all finite subsets (resp., all finite nonempty subsets) $X$ of $M$.  

\begin{example}  The set $\K(2^M)$ of all finite subsets of a magma $M$ is a sub prequantic semilattice of the prequantale $2^M$ and is the free prequantic semilattice on the magma $M$, and the set $\K(2^M) \setmin \{\emptyset\}$ of all finite nonempty subsets of $M$ is a sub multiplicative semilattice of $\K(2^M)$ and is the free multiplicative semilattice on $M$.
\end{example}

\begin{proposition}\label{semiquantle}
The following are equivalent for any ordered magma $M$.
\begin{enumerate}
\item $M$ is a prequantic semilattice (resp., multiplicative semilattice).
\item $a (\bigvee X) = \bigvee (a X)$ and $(\bigvee X) a = \bigvee(Xa)$ for any $a \in M$ and any finite subset (resp., any finite nonempty subset) $X$ of $M$.
\item The map $\K(2^M) \longrightarrow M$ (resp., $\K(2^M)\setmin\{\emptyset\} \longrightarrow M$) acting by $X \longmapsto \bigvee X$ is a well-defined magma homomorphism.
\item $\bigvee (X Y) = \bigvee X \bigvee Y$ for all finite subsets (resp., all finite nonempty subsets) $X$ and $Y$ of $M$.
\end{enumerate}
\end{proposition}




\begin{proposition}\label{preq}
A prequantale (resp., near prequantale, semiprequantale) is equivalently a complete (resp., near sup-complete, bounded complete) Scott-topological prequantic semilattice (resp., multiplicative semilattice, multiplicative semilattice).  
\end{proposition}

\section{Nuclei}\label{sec:MCO}

A {\it preclosure} on a poset $S$ is a self-map $\star$ of $S$ that is expansive and order-preserving.  A closure operation is equivalently an idempotent preclosure.  

\begin{lemma}\label{starlemma}
Let $\star$ be a closure operation on a poset $S$ and let $X \subseteq S$.
\begin{enumerate}
\item $\bigvee_{S^\star} X^\star = (\bigvee_S X^\star)^\star = (\bigvee_S X)^\star$ if $\bigvee_S X$ exists.
\item $\bigwedge_{S^\star} X^\star =  (\bigwedge_S X^\star)^\star = \bigwedge_S X^\star$ if $\bigwedge_S X^\star$ exists.
\item The map $\star: S \longrightarrow S^\star$ is sup-preserving.
\item If $S$ is complete (resp., near sup-complete, bounded complete, directed complete, a bdcpo, a join semilattice, a meet semilattice), then so is $S^\star$.
\end{enumerate}
\end{lemma}

If $\star$ is a nucleus on an ordered magma $M$, then the set $M^\star$ is an ordered magma under $\star$-multiplication, the corestriction ${_{M^\star}}|\star : M \longrightarrow M^\star$ is a sup-preserving morphism of ordered magmas, and if $1$ is an identity element of $M$ then $1^\star$ is an identity element of $M^\star$.

The following elementary results give several characterizations of nuclei.

\begin{proposition}\label{closureprop1}
The following conditions are equivalent for any closure operation $\cl$ on an ordered magma $M$.
\begin{enumerate}
\item $\cl$ is a nucleus on $M$.\item $x y^\cl \leq (xy)^\cl$ and  $x^\cl y \leq (xy)^\cl$ for all $x,y \in M$.
\item $(x^\cl y^\cl)^\cl = (xy)^\cl$ for all $x,y \in M$, or equivalently, the map $\star: M \longrightarrow M^\star$ is a magma homomorphism.
\end{enumerate}
If $M$ is an ordered monoid, then the above conditions are equivalent to the following.
\begin{enumerate}
\item[(4)] $\cl$-multiplication on $M$ is associative.
\end{enumerate}
\end{proposition}

\begin{proposition}\label{closureprop1a}
The following are equivalent for any self-map $\cl$ of a left or right unital  ordered magma $M$.
\begin{enumerate}
\item $\cl$ is a nucleus on $M$.
\item $xy \leq z^\cl  \Leftrightarrow  x y^\cl \leq z^\cl  \Leftrightarrow x^\cl y \leq z^\cl$ for all $x,y,z \in M$.
\item $x \leq x^\star$, and $xy \leq z^\cl \Rightarrow x^\cl y^\cl \leq z^\cl$, for all $x,y,z \in M$.
\end{enumerate}
If $M$ is left or right unital and near residuated, then the above conditions are equivalent to the following.
\begin{enumerate}
\item[(4)] $x^\star/y = x^\star/y^\star$ (resp., $y \backslash x^\star = y^\star \backslash x^\star$) for all $x,y \in M$ such that $x^\star/y$ or  $x^\star/y^\star$ (resp., $y \backslash x^\star$ or  $y^\star \backslash x^\star$) is defined.
\end{enumerate}
\end{proposition}

We will say that a subset $\Sigma$ of an ordered magma $M$ is a {\it sup-spanning subset} of $M$ if $$xy = \bigvee\{ay: a \in \Sigma \mbox{ and } a \leq x\} = \bigvee\{xb: b \in \Sigma \mbox{ and } b \leq y\}$$ for all $x,y \in M$.   If $M$ is unital and residuated, or more generally if $M$ is left or right unital and for all $a \in M$ the left and right multiplication by $a$ maps on $M$ are sup-preserving (resp., near sup-preserving), then $\Sigma$ is a sup-spanning subset of $M$ if and only if every element of $M$ is the supremum of some subset of $\Sigma$ (resp., every element of $M$ is the supremum of some subset of $\Sigma$ and $\bigwedge M$ is an annihilator of $M$ if $\bigwedge M$ exists and is not in $\Sigma$).  For example, the set $\operatorname{Prin}(D)$ of all nonzero principal fractional ideals of an integral domain $D$ is a sup-spanning subset of the residuated near quantale $\F(D)$ of all nonzero $D$-submodules of the quotient field of $D$.


\begin{proposition}\label{closureprop2}
Let $\cl$ be a closure operation on an ordered magma $M$ and let $\Sigma$ be any sup-spanning subset of $M$.  Then $\cl$ is a nucleus on $M$ if and only if $ax^\cl \leq (ax)^\cl$ and $x^\cl a \leq (xa)^\cl$ for all $a \in \Sigma$ and all $x \in M$, and in that case $\Sigma^\cl$ is a sup-spanning subset of $M^\cl$.
\end{proposition}

\begin{proof}
Necessity of the given condition is clear.  Suppose that $ax^\cl \leq (ax)^\cl$ and $x^\cl a \leq (xa)^\cl$ for all $a \in \Sigma$ and all $x \in M$.  Then
$xy^\cl = \bigvee_{a \in \Sigma: \ a \leq x} ay^\cl \leq \left(\bigvee_{a \in \Sigma: \ a \leq x} ay\right)^\cl = (xy)^\cl,$ and by symmetry $x^\cl y \leq (xy)^\cl$, for all $x,y \in M$, so that $\cl$ is a nucleus.  That $\Sigma^\star$ is a sup-spanning subset of $M^\star$ follows easily from Lemma \ref{starlemma}(1).
\end{proof}

For any self-map $\star$ of a magma $M$ we say that $a \in M$ is {\it transportable through $\star$} if $(ax)^\star = ax^\star$ and $(xa)^\star = x^\star a$ for all $x \in M$, and we let $\T^\star(M)$ denote the set of all elements of $M$ that are transportable through $\star$.

\begin{corollary}\label{joinspan}
Any closure operation $\star$ on an ordered magma $M$ such that $\T^\star(M)$ is a sup-spanning subset of $M$ is a nucleus on $M$.
\end{corollary}

For any magma $M$ we let $\Inv(M)$ denote the set of all $u \in M$ for which there exists $u^{-1} \in M$ such that the left and right multiplication by $u^{-1}$ maps are inverses, respectively, to the left and right multiplication by $u$ maps.
For any ordered magma $M$ we let $\U(M)$ denote the set of all $u \in M$ such that the left and right multiplication by $u$ maps are poset automorphisms of $M$.  One has $\Inv(M) \subseteq \U(M)$ for any ordered magma $M$, with equality holding if $M$ is a monoid.

\begin{proposition}\label{closureprop3}
One has $\Inv(M) \subseteq \T^\star(M)$ for any nucleus $\cl$ on an ordered magma $M$.
\end{proposition}

\begin{proof}
If $u \in \Inv(M)$, then $x^\cl = u^{-1}(u x^\cl) \leq u^{-1}(ux)^\cl \leq (u^{-1}(ux))^\cl   = x^\cl$, whence $ux^\cl = (ux)^\cl$, for all $x \in M$.  By symmetry one has $u \in \T^\star(M)$.
\end{proof}


\begin{remark}
Let $D$ be an integral domain with quotient field $F$.  By definition, a semistar operation on $D$ is a closure operation $\star$ on $\F(D)$ such that $(aI)^\star = aI^\star$ for all $I \in \F(D)$ and all nonzero $a \in F$, that is, such that $\operatorname{Prin}(D) \subseteq \T^\star(\F(D))$.  Since $\operatorname{Prin}(D)$ is a sup-spanning subset of $\F(D)$, Corollary \ref{joinspan} implies that any semistar operation on $D$ is a nucleus on $\F(D)$.  Conversely, since $\operatorname{Prin}(D) \subseteq \Inv(\F(D))$ (and in fact $\Inv(\F(D))$ is the set of invertible fractional ideals of $D$), Proposition \ref{closureprop3} implies that any nucleus on $\F(D)$ is a semistar operation on $D$.  A similar argument shows that a star operation on $D$ is equivalently a nucleus $*$ on $\Fp(D)$ such that $D^* = D$.
\end{remark}

We will say that a {\it $\U$-lattice} (resp., {\it near $\U$-lattice}, {\it semi-$\U$-lattice}) is a sup-magma (resp., near sup-magma, ordered magma that is a bounded complete join semilattice) $M$ such that $\U(M)$ is a sup-spanning subset of $M$.  

\begin{example} For any integral domain $D$ with quotient field $F$, the multiplicative lattice $\mbox{Mod}_D(F)$ of all $D$-submodules of $F$ is a $\U$-lattice, and the near multiplicative lattice $\F(D)$ of all nonzero $D$-submodules of $F$ is a near $\U$-lattice.
\end{example}

\begin{proposition} An ordered magma $M$ is a prequantale (resp., near prequantale, semiprequantale) if and only if $M$ is complete (resp., near sup-complete, a bounded complete join semilattice) and the set of all $a \in M$ such that the left and right multiplication by $a$ maps on $M$ are sup-preserving (resp., near sup-preserving, near sup-preserving) is a sup-spanning subset of $M$.
\end{proposition}

\begin{corollary}\label{ulattices}
Any $\U$-lattice (resp., near $\U$-lattice, semi-$\U$-lattice) is a prequantale (resp.,  near prequantale, semiprequantale).  
\end{corollary}

\begin{proposition}\label{CSTstar}
Let $\star$ be a nucleus on an ordered magma $M$.   If $M$ is a prequantale (resp., near prequantale, residuated, near residuated, semiprequantale, quantale, near quantale, semiquantale, multiplicative lattice, near multiplicative lattice, semimultiplicative lattice, $\U$-lattice, near $\U$-lattice, semi-$\U$-lattice, multiplicative semilattice, prequantic semilattice), then so is $M^\star$.  Moreover, if $M$ is near residuated, then $(x/y)^\star = x/y^\star = x/y$ (resp., $(y \backslash x)^\star = y^\star \backslash x = y \backslash x$) for all $x \in M^\star$ and all $y \in M$ such that $x/y$ (resp., $y \backslash x$) is defined.
\end{proposition}

\begin{proof}  Suppose that $M$ is a prequantale.  By Lemma \ref{starlemma}(4) the partial ordering  on $M^\star$ is complete.  For any $a \in M^\star$ and $X \subseteq M^\star$ we have $a \star  {\bigvee}_{M^\star} X = a \star ({\bigvee}_M X)^\star = (a \ {\bigvee}_M X)^\star = ( {\bigvee}_M a X)^\star = ({\bigvee}_M (a X)^\star)^\star  = ({\bigvee}_M (a \star X))^\star =     {\bigvee}_{M^\star} (a \star X)$, and therefore $a \star {\bigvee}_{M^\star} X = {\bigvee}_{M^\star} (a \star X)$.   By symmetry the corresponding equation holds for right multiplication.  Thus $M^\star$ is a prequantale.  The proof of the rest of the first statement of the proposition is similar.  The second statement of the proposition follows from the equivalences $yz \leq x^\star \Leftrightarrow y^\star z \leq x^\star \Leftrightarrow y z^\star \leq x^\star \Leftrightarrow y^\star z^\star \leq x^\star$.
\end{proof}

By following proposition, which is a straightforward generalization of \cite[Theorem 3.1.1]{ros}, the nuclei on a near prequantale or prequantale $Q$ classify the quotient objects of $Q$.

\begin{proposition}\label{supremark}
Let $f: Q \rightarrow M$ be a morphism of near sup-magmas (resp., sup-magmas), where $Q$ is a near prequantale (resp., prequantale).
\begin{enumerate}
\item If $\star$ is any nucleus on $Q$, then ${_{Q^\star}}|\star: Q \longrightarrow Q^\star$ is a surjective morphism of near prequantales (resp., prequantales).
\item There exists a unique nucleus $\star$ on $Q$ such that $f = (f|_{Q^\star}) \circ ({_{Q^\star}}|\star)$ and $f|_{Q^\star}$ is injective; moreover, one has $x^\star = \bigvee\{y \in Q: \ f(y) = f(x)\}$ for all $x \in Q$.
\item $f|_{Q^\star}: Q^\star \longrightarrow M$ is an embedding of near sup-magmas (resp., sup-magmas).
\item $f|_{Q^\star}: Q^\star \longrightarrow \operatorname{im} f$ is an isomorphism of ordered magmas.
\end{enumerate}
\end{proposition}
 
A quantale $Q$ is {\it simple} if every nontrivial sup-preserving semigroup homomorphism from $Q$ is injective.  More generally, we will say that a near sup-magma $M$ is {\it simple} if every nonconstant near sup-preserving magma homomorphism from $M$ is injective.  If $M$ is a sup-magma, then this holds if and only if every nontrivial sup-preserving magma homorphism from $M$ is injective, so our definition is consistent with that of a simple quantale.  Let $d$ be the identity map on $M$, and let $e: M \longrightarrow M$ be defined by $x^e = \bigvee M$ for all $x \in M$.  Both $d$ and $e$ are nuclei.

\begin{corollary}\label{desimple}
A near prequantale $Q$ is simple if and only if $d$ and $e$ are the only nuclei on $Q$.
\end{corollary}



\section{The poset of nuclei}\label{sec:COCL}

Let $S$ be a poset and $X$ a set.  The set $S^X$ of all functions from $X$ to $S$ is partially ordered, where $f \leq g$ if $f(x) \leq g(x)$ for all $x \in X$.  If $f \leq g$, then we say that $f$ is {\it smaller}, or {\it finer}, than $g$, or equivalently $g$ is {\it larger}, or {\it coarser}, than $f$.  The set $\C(S)$ of all closure operations on $S$ inherits a partial ordering from the poset $S^S$.  For any $\star_1, \star_2 \in \C(S)$ one has $\star_1 \leq \star_2$ if and only if $S^{\star_1} \supseteq S^{\star_2}$.  The identity operation $d$ on $S$ is the smallest closure operation on $S$.  If $\bigvee S$ exists then there is a largest closure operation $e$ on $S$, given by $x^e = \bigvee S$ for all $x \in S$.   If $M$ is an ordered magma, then we let $\N(M)$ denote the subposet of $\C(M)$ consisting of all nuclei on $M$.   In this section we study properties of the poset $\N(M)$. 

The following result characterizes nuclei on near residuated ordered magmas and generalizes \cite[Lemma 3.33]{gal}.

\begin{proposition}\label{characterizingclosures2}
Let $C$ be a subset of a poset $S$.
\begin{enumerate}
\item There exists a closure operation $\star$ on $S$ with $C = S^\star$ if and only if $\bigwedge\{a \in C: a \geq x\}$ exists in $S$ for all $x \in S$.  For any such closure operation $\star$ one has $x^\star = \bigwedge\{a \in C: a \geq x\}$ for all $x \in S$, and therefore $\star = \star_C$ is uniquely determined by $C$. 
\item If $S = M$ is a near residuated ordered magma and $\star_C$ exists, then $\star_C$ is a nucleus on $M$ if and only, for all $x \in C$ and $y \in M$, one has $x/y \in C$ provided $x/y$ exists and $y \backslash x \in C$ provided $y \backslash x$ exists.
\end{enumerate}
\end{proposition}

\begin{proof} Statement (1) is  well-known and easy to check.
We prove (2). If $\star$ is a nucleus, then it follows from Proposition \ref{CSTstar} that $x/y \in C$ and $y \backslash x \in C$ for all $x,y$ as in statement (2).  Conversely, suppose that this condition on $C$ holds.  Let $x, y \in M$.  Then we claim that $x^\star y =  \bigwedge\{a \in C: \ a \geq x\} y$ is less than or equal to $(xy)^\star = \bigwedge\{b \in C: \ b \geq xy\}$.  For let $b \in C$ with $b \geq xy$.  Set $a = b/y$, whence $a \geq x$.  By hypothesis one has $a \in C$.  Therefore $x^\star y \leq ay = (b/y)y \leq b$.  Taking the infimum over all such $b$, we see that $x^\star y \leq \bigwedge\{b \in C: \ b \geq xy\} = (xy)^\star$.  By symmetry one also has $xy^\star \leq (xy)^\star$.  It follows that $\star$ is a nucleus.
\end{proof}


\begin{corollary}\label{characterizingclosures}
One has the following.
\begin{enumerate}
\item Let $C$ be a nonempty subset of a bounded complete poset $S$.  Then there exists a closure operation $\star$ on $S$ with $C = S^\star$ if and only if $\bigwedge X \in C$ for all nonempty $X \subseteq C$ bounded below.
\item Let $C$ be a nonempty subset of a bounded complete and near residuated ordered magma $Q$.  Then there exists a nucleus $\star$ on $Q$ with $C = Q^\star$ if and only if $\bigwedge X \in C$ for all nonempty $X \subseteq C$ bounded below and  for all $x \in C$ and $y \in Q$ one has $x/y \in C$ provided $x/y$ exists and $y \backslash x \in C$ provided $y \backslash x$ exists.
\end{enumerate}
\end{corollary}

\begin{corollary}\label{charclos}
One has the following.
\begin{enumerate}
\item Let $S$ be a bounded complete poset.  Then $\C(S)$ is bounded complete.  Moreover, one has $S^{\bigvee \Gamma} = \bigcap_{\star \in \Gamma} S^\star$ for all bounded subsets $\Gamma$ of $\C(S)$.
\item Let $Q$ be a bounded complete and near residuated ordered magma.  Then $\N(M)$ is bounded complete.  Moreover, one has $Q^{\bigvee \Gamma} = \bigcap_{\star \in \Gamma} Q^\star$ for all bounded subsets $\Gamma$ of $\N(Q)$.
\end{enumerate}
\end{corollary}

\begin{proof}
One easily checks that the intersection of nonempty subsets $C$ of $S$ that satisfy the condition of statement (1) of Corollary \ref{characterizingclosures} itself satisfies the same condition and is nonempty given the boundedness condition on $\Gamma$.  Therefore, since $S^{\star} \subseteq S^{\star'}$ if and only if $\star \geq \star'$ for all $\star, \star' \in \C(S)$,  statement (1) follows.  Statement (2) is proved in a similar fashion.
\end{proof}

For any self-map $\star$ of a set $S$ we let $\Fix(\star) = \{x \in S: \ x^\star = x\}$.

\begin{lemma}\label{preclosurelemma}
Let $S$ be a bounded complete poset, and let $+$ be a preclosure on $S$ that is bounded above by some closure operation on $S$.
\begin{enumerate}
\item There exists a smallest closure operation $\star$ on $S$ that is larger that $+$, and one has $S^\star = \Fix(+)$ and $x^\star = \bigwedge\{y \in \Fix(+): \ y \geq x\}$ for all $x \in S$.
\item Define $x^{+_0} = x$ and $x^{+_\alpha} = \bigvee\{(x^{+_\beta})^+: \ \beta < \alpha\}$ for all $x \in S$ and all (finite or transfinite) ordinals $\alpha$.  One has $+_\alpha = \star$ for all $\alpha \gg 0$.  
\item If $S = M$ is a bounded complete and near residuated ordered magma and $xy^+ \leq (xy)^+$ and $x^+ y \leq (xy)^+$ for all $x,y \in M$, then $\star$ is a nucleus on $M$.
\end{enumerate}
\end{lemma}

\begin{proof} \
\begin{enumerate}
\item The set $\{y \in \Fix(+): \ y \geq x\}$ is nonempty and bounded by the hypothesis on $+$. Define $x^\star = \bigwedge\{y \in \Fix(+): \ y \geq x\}$ for all $x \in S$.  Clearly $\star$ is a preclosure on $S$ with $\Fix(+) \subseteq S^\star$.  For the reverse inclusion note that
$$(x^+)^\star = \bigwedge\{y  \in \Fix(+): \ y \geq x^+\} = \bigwedge\{y \in \Fix(+): \ y \geq x\} =  x^\star$$
and therefore $x \leq x^+ \leq (x^+)^\star = x^\star$, whence $S^\star \subseteq \Fix(+)$. 
Therefore
$$(x^\star)^\star  =  \bigwedge\{y  \in \Fix(+): \ y \geq x^\star\} = \bigwedge\{y \in \Fix(+): \ y \geq x\} = x^\star,$$
whence $\star$ is a closure operation on $S$.  It is then clear that $\star$ is the smallest closure operation on $S$ that is larger than $+$.
\item This follows readily from statement (1).  
\item  Let $a,x \in M$.  Let $z$ be any element of $M$ such that $z \geq ax$ and $z^+ = z$.  Set $y = a \backslash z$.  Then $y \geq x$ and $y^+ = (a \backslash z)^+ \leq a \backslash z^+ = a \backslash z = y$, hence $y^+ = y$.  Moreover, we have $ay \leq z$.
Therefore we have $$ax^\star \leq \bigwedge\{ay: \ y \geq x, \ y^+ = y\} 
  \leq \bigwedge\{z \in M: z \geq ax, \  z^+ = z\}
  = (ax)^\star.$$
By symmetry we also have $x^\star a \leq (xa)^\star$, whence $\star$ is a nucleus.
\end{enumerate}
\end{proof}

Let $S$ be a poset and $\Gamma \subseteq \C(S)$.  Define partial self-maps $\bigsqcapp \Gamma$ and $\bigsqcupp \Gamma$ of $S$ by $$x^{\bigsqcapp\Gamma} = \bigwedge\{x^\star: \ \star \in \Gamma\},$$
$$x^{\bigsqcupp\Gamma} = \bigwedge\{y \in S: \ y \geq x \mbox{ and } \forall \star \in \Gamma \ (y^\star = y)\},$$ respectively, for all $x \in S$ such that the respective infima exist.  

\begin{lemma}\label{closelem}
Let $S$ be a near sup-lattice (resp., bounded complete poset).  Then $\C(S)$ is a complete lattice (resp., bounded complete poset), and one has the following.
\begin{enumerate}
\item $\bigwedge_{\C(S)} \Gamma = \bigsqcapp \Gamma$ and $S^{\bigsqcapp \Gamma} \supseteq \bigcup_{\star \in \Gamma} S^\star$ for all subsets (resp., all nonempty subsets) $\Gamma$ of $\C(S)$.
\item $\bigvee_{\C(S)} \Gamma = \bigsqcupp \Gamma$ and $S^{\bigsqcupp\Gamma} = \bigcap_{\star \in \Gamma} S^\star$ for all subsets (resp., bounded subsets) $\Gamma$ of $\C(S)$.
\end{enumerate}
\end{lemma}

\begin{proof} \
\begin{enumerate}
\item  $\bigsqcapp \Gamma$ is a preclosure on $S$, and for all $x \in S$ one has $(x^{\bigsqcapp\Gamma})^{\bigsqcapp\Gamma}  = \bigwedge_{\star \in \Gamma} (x^{\bigsqcapp\Gamma})^\star \leq  \bigwedge_{\star \in \Gamma} x^\star = x^{\bigsqcapp\Gamma}.$
Therefore $\bigsqcapp\Gamma$ is a closure operation on $S$, from which it follows that $\bigsqcapp\Gamma = \bigwedge_{\C(S)} \Gamma$.   The given inclusion follows immediately from the definition of $\bigsqcapp \Gamma$.
\item Define $x^+ = \bigvee\{x^\star: \ \star  \in \Gamma\}$ for all $x \in S$, which exists since the $x^\star$ are  bounded above by $x^{\bigvee_{\C(S)} \Gamma}$.  Clearly $+$ is a preclosure on $S$ bounded above by $\bigvee_{\C(S)} \Gamma$, and one has $x^+ = x$ if and only if $x^\star = x$ for all $\star \in \Gamma$.  Therefore $\bigsqcupp \Gamma$ is a closure operation on $S$ by Lemma \ref{preclosurelemma}(1), and it follows easily that $\bigsqcupp \Gamma = \bigvee_{\C(S)} \Gamma$.   Finally, that $S^{\bigsqcupp\Gamma} = \bigcap_{\star \in \Gamma} S^\star$  follows from Corollary \ref{charclos}(1).
\end{enumerate}
\end{proof}

The following result generalizes \cite[Proposition 3.1.3]{ros} and \cite[Example 1.5]{fon2}.

\begin{proposition}\label{CMC}
Let $M$ be a near sup-magma (resp., bounded complete ordered magma).  Then $\N(M)$ is a complete lattice (resp., bounded complete poset), and one has the following. 
\begin{enumerate}
\item $\bigwedge_{\N(M)} \Gamma = \bigsqcapp \Gamma$ and $M^{\bigsqcapp \Gamma} \supseteq \bigcup_{\star \in \Gamma} M^\star$ for all subsets (resp., all nonempty subsets) $\Gamma$ of $\N(M)$.
\item $\bigvee_{\N(M)} \Gamma = \bigsqcupp\Gamma$ and $M^{\bigsqcupp \Gamma} = \bigcap_{\star \in \Gamma} M^\star$ for all subsets (resp., all bounded subsets) $\Gamma$ of $\N(M)$ if $M$ is a near prequantale (resp., bounded complete and near residuated).
\end{enumerate}
\end{proposition}

\begin{proof} \
\begin{enumerate}
\item By Lemma \ref{closelem} it suffices to show that $\bigsqcapp \Gamma$ is a nucleus.  We have
$$xy^{\bigsqcapp\Gamma}   =  x\bigwedge_{\star \in \Gamma} y^\star \leq  \bigwedge_{\star \in \Gamma} xy^\star \leq \bigwedge_{\star \in \Gamma} (xy)^\star = (xy)^{\bigsqcapp\Gamma},$$ and similarly $x^{\bigsqcapp\Gamma}y \leq  (xy)^{\bigsqcapp\Gamma}$, for all $x,y \in M$, whence $\bigsqcapp\Gamma$ is a nucleus.
\item By Lemma \ref{closelem} it suffices to show that $\bigsqcupp \Gamma$ is a nucleus.  Since $d = \bigsqcupp \emptyset$ is a nucleus we may assume that $\Gamma$ is nonempty.  Let $+$ be the preclosure on $M$ defined in the proof of Lemma \ref{closelem}(2).  For all $x,y \in M$ we have
$$xy^+ = x\bigvee_{\star \in \Gamma} y^\star = \bigvee_{\star \in \Gamma} xy^\star \leq \bigvee_{\star \in \Gamma} (xy)^\star = (xy)^+,$$ and similarly $x^+ y \leq (xy)^+$.  Therefore $\star$ is a nucleus by Lemma \ref{preclosurelemma}(3).
\end{enumerate}
\end{proof}



Note that, if an ordered magma $M$ is a meet semilattice, then $\N(M)$ is also a meet semilattice, and statement (1) of Proposition \ref{CMC} holds for all nonempty finite subsets $\Gamma$ of $M$.  The proof is similar to that of Proposition \ref{CMC}.

If $N$ is a submagma of a near sup-magma $M$, then for any nucleus $\star$ on $N$ we may define
$$\ind_M(\star) = {\bigwedge}_{\N(M)} \{\star' \in \N(M):  \ \star'|_N = \star\},$$
$$\ind^M(\star) = {\bigvee}_{\N(M)} \{\star' \in \N(M):  \ \star'|_N = \star\},$$  both of which are nuclei on $M$.  We say that a subset $X$ of an ordered magma $M$ is {\it saturated} if whenever $xy \in X$ for some $x,y \in M$ that are not annihilators of $M$ one has $x \in X$ and $y \in X$.

\begin{proposition}\label{extendingclosures}
Let $M$ be a near sup-magma and $\star$ a nucleus on a submagma $N$ of $M$.
\begin{enumerate}
\item Suppose that $\{\star' \in \N(M): \ \star'|_N = \star\}$ is nonempty.  Then  $\ind_M(\star)$ is the smallest nucleus on $M$ whose restriction to $N$ is equal to $\star$.  If $M$ is bounded complete and near residuated, then $\ind^M(\star)$ is the largest nucleus on $M$ whose restriction to $N$ is equal to $\star$, and one has $\star'|_N = \star$ if and only if $\ind_M(\star) \leq \star' \leq \ind^M(\star)$.
\item If $M$ is a near prequantale and $N$ is a sup-spanning subset of $M$, then
$$x^{\ind_M(\star)} = \bigwedge \{y \in M: \ y \geq x \mbox{ and } \forall z \in N \ (z \leq y \Rightarrow z^\star \leq y)\}$$
for all $x \in M$, and $\ind_M(\star)|_N = \star$.
\item If $N$ is a saturated downward closed subset of $M$ and $M$ is bounded above, then $$x^{\ind^M(\star)} = \left\{\begin{array}{ll} x^\star & \mbox{if } x \in N \\ \bigvee M & \mbox{otherwise}\end{array}\right.$$ 
for all $x \in M$, and $\ind^M(\star)|_N = \star$.
\end{enumerate}
\end{proposition}

\begin{proof} \
\begin{enumerate}
\item This follows easily from Proposition \ref{CMC}.
\item First we define $x^+ = \bigvee\{y^\star: \ y \in N \mbox{ and } y \leq x\}$ for all $x \in M$.  As $N$ is a sup-spanning subset of $M$ one has $x \leq x^+$ for all $x \in M$.  Since $+$ is order-preserving it follows that $+$ is a preclosure on $M$.  One has $y^+ = y$ if and only if $\forall z \in N \ (z \leq y \Rightarrow z^\star \leq y)$, whence by Lemma \ref{preclosurelemma}(1) it follows that the operation $\star_M$ on $M$ defined by
$$x^{\star_M} = \bigwedge \{y \in M: \ y \geq x \mbox{ and } \forall z \in N \ (z \leq y \Rightarrow z^\star \leq y)\}$$ is a closure operation on $M$.
Now, for all $a \in N$ and $x \in M$ one has
$$ay^+ = \bigvee_{z \in N: \ z \leq y} az^\star
 \leq \bigvee_{z \in N: \ z \leq y} (az)^\star
 \leq \bigvee_{w \in N: \ w \leq ay} w^\star
 = (ay)^+.$$
Since $N$ is a sup-spanning subset of $M$, it follows that
$xy^+ \leq (xy)^+$, and by symmetry $x^+y \leq (xy)^+$, for all $x,y \in M$.  Thus $\star_M$ is a nucleus by Lemma \ref{preclosurelemma}(3).  Now, $+$ is the smallest preclosure on $M$ whose restriction to $N$ is $\star$, and by Lemma \ref{preclosurelemma}(1) the operation $\star_M$ is the smallest closure operation on $M$ that is larger than $+$, and one has $(\star_M)|_N = \star$.  It follows that $\star_M$ is the smallest closure operation on $M$ whose restriction to $N$ is $\star$.  Therefore, since $\star_M$ is a nucleus we must have $\star_M = \ind_M(\star)$ by statement (1). 
\item It suffices to show that the operation defined in the statement is a nucleus.  This is easy to check.
\end{enumerate}
\end{proof}

\begin{remark}
Let $\star_1$ and $\star_2$ be nuclei on an ordered magma $M$.  Then $\star_1  \vee \star_2$ exists in $\N(M)$ and equals the $n$-fold composition $\star = \cdots \circ \star_2 \circ \star_1 \circ \star_2$ for some positive integer $n$ if and only if $\star$ is larger than the $n$-fold composition $\cdots \circ \star_1 \circ \star_2 \circ \star_1$.   In particular, $\star_1 \vee \star_2 = \star_1 \circ \star_2$ for two nuclei $\star_1$ and $\star_2$ if and only if $\star_1  \circ \star_2 \geq \star_2 \circ \star_1$.   The operation $\circ$ on closure operations, though not necessarily closed, is studied in \cite[Section 2]{pic} and \cite{vas}.
\end{remark}

\section{Divisorial closure operations}\label{sec:DC}

In this section we generalize the generalized divisorial closure semistar operations \cite[Example 1.8(2)]{pic}, and we use this to derive a characterization of the simple near prequantales that is of a different vein than the characterization in \cite[Theorem 2.5]{krum} of the simple quantales. 

 If $a$ is an element of an ordered magma $M$ and there exists a largest nucleus $\star$ on $M$ such that $a^\star = a$, then we denote it by $v(a) = v_M(a)$ and call it {\it divisorial closure on $M$ with respect to $a$}.  For example, if $\bigvee M$ exists then $v(\bigvee M)$ exists and equals $e = \bigvee \N(M)$.   The proof of the following result is straightforward.

\begin{proposition}\label{divprop}
Let $M$ be an ordered magma.
\begin{enumerate}
\item Suppose that $v(a)$ exists, where $a \in M$.  Then, for all $\star \in \N(M)$, one has $a^\star = a$ if and only if $\star \leq v(a)$. 
\item Let $a \in M$.  Then $v(a)$ exists if and only if $v' = \bigvee \{\star \in \N(M): \ a^\star = a\}$ exists in $\N(M)$ and $a^{v'} = a$, in which case $v(a) = v'$.
\item If $\star$ is a nucleus on $M$ such that $v(a)$ exists for all $a \in M^\star$, then $\star = \bigwedge \{v(a): \ a \in M^\star\}$.
\end{enumerate}
\end{proposition}


The next proposition follows from Propositions \ref{divprop} and \ref{CMC}(2).

\begin{proposition}\label{vclosureprop}
Let $Q$ be a near prequantale.  Then $v(a)$ exists and equals $\bigvee \{\star \in \N(Q): \ a^\star = a\}$ for all $a \in Q$, and one has $\star = \bigwedge \{v(a): \ a \in Q^\star\}$ for all $\star \in \N(Q)$.
\end{proposition}

\begin{corollary}\label{simpleprequantales}
The following are equivalent for any near prequantale $Q$.
\begin{enumerate}
\item $Q$ is simple.
\item The only nuclei on $Q$ are $d$ and $e$.
\item $v(a) = d$ for all $a < \bigvee Q$.
\end{enumerate}
\end{corollary}

The divisorial closure operations generalize as follows.   If $S$ is a subset of an ordered magma $M$ and there exists a largest nucleus $\star$ on $M$ such that $S \subseteq M^\star$, then we denote it by $v(S) = v_M(S)$ and call it {\it divisorial closure on $M$ with respect to $S$}.  Note that $v(a) = v(\{a\})$.  Proposition \ref{divprop} generalizes as follows.

\begin{proposition}\label{divprop2}
Let $S$ be a subset of an ordered magma $M$.
\begin{enumerate}
\item Suppose that $v(S)$ exists.  Then, for all $\star \in \N(M)$, one has $S \subseteq M^\star$ if and only if $\star \leq v(S)$. 
\item The following conditions are equivalent.
\begin{enumerate}
\item $v(S)$ exists.
\item $v' = \bigvee \{\star \in \N(M): \ S \subseteq M^\star\}$ exists in $\N(M)$ and $S \subseteq M^{v'}$.
\item There is a smallest subset $C$ of $M$ containing $S$ such that $C = M^\star$ for some $\star \in \N(M)$.
\end{enumerate}
Moreover, if the above condtions hold, then $v' = \star = v(S)$.
\item Suppose that $S = \bigcup_\lambda T_\lambda$ and $v(T_\lambda)$ exists for all $\lambda$.  Then $v(S)$ exists if and only if $\bigwedge_\lambda v(T_\lambda)$ exists, in which case they are equal.
\item $\star = v(M^\star)$ for any nucleus $\star$ on $M$.
\end{enumerate}
\end{proposition}

\begin{proposition}
Let $Q$ be a near prequantale. Then $v(S)$ exists for any subset $S$ of $Q$ and equals $\bigwedge \{v(a): a \in S\}$.
\end{proposition}

\begin{proposition}\label{vgen}
Let $Q$ be a bounded complete and near residuated ordered magma.  If $v(S)$ exists for some subset $S$ of $Q$, then $v(T)$ exists for any subset $T$ of $Q$ containing $S$.
\end{proposition}

If $Q$ is a unital near quantale then we can determine an explicit formula for the divisorial closure operations $v(a)$.

\begin{lemma}\label{vlemma}
Let $M$ be an ordered monoid and $a \in M$.  If for any $x \in M$ there is a largest element $x^\star$ of $M$ such that $rxs \leq a$ implies $r x^\star s \leq a$ for all $r,s \in M$, then $v(a)$ exists and equals $\star$.
\end{lemma}

\begin{proof}
First we show that the map $\star$ defined in the statement of the lemma is a nucleus on $M$.  Let $x,y \in M$.  If $x \leq y$, then $rys \leq a \Rightarrow rxs \leq a \Rightarrow rx^{\star}s \leq a$, whence $x^{\star} \leq y^{\star}$.   Likewise $rxs \leq a \Rightarrow rxs \leq a$, whence $x \leq x^{\star}$.  Moreover, one has
$rxs \leq a \Rightarrow rx^{\star} s \leq a \Rightarrow r(x^{\star})^{\star}s \leq a$, whence $(x^{\star})^{\star} \leq x^{\star}$, whence equality holds.  Thus $\star$ is a closure operation on $M$.  Next, note that $rxys \leq a \Rightarrow rx^{\star}ys \leq a \Rightarrow rx^{\star}y^{\star}s \leq a$, whence $x^{\star}y^{\star} \leq (xy)^{\star}$.  Thus $\star$ is a nucleus on $M$.

Next we observe that $1a1 \leq a$, which implies $1a^{\star}1 \leq a$, whence $a^{\star} = a$.
Finally, let $\star'$ be any nucleus on $M$ with $a^{\star'} = a$.  Then for any $x \in M$ we have
$rxs \leq a \Rightarrow rx^{\star'}s \leq a^\star = a$ for all $r,s \in M$, whence $x^{\star'} \leq x^\star$.  Thus $\star' \leq \star$.  It follows that $\star$ is the largest nucleus on $M$ such that $a^{\star} = a$, so $v(a)$ exists and equals $\star$.
\end{proof}

\begin{proposition}\label{vquantales}
Let $Q$ be a unital near quantale and $a \in Q$.  Then one has
$$x^{v(a)} =  \bigvee\{y \in Q: \ \forall r,s \in Q \ (rxs \leq a \Rightarrow rys \leq a)\}$$
for all $x \in Q$; alternatively, $x^{v(a)}$ is the largest element of $Q$ such
that $rxs \leq a$ implies $rx^{v(a)}s \leq a$ for all $r,s \in Q$.
\end{proposition}

\begin{proof}
Let $x^\star =  \bigvee\{y \in Q: \ \forall r,s \in Q \ (rxs \leq a \Rightarrow rys \leq a)\}$ for all $x \in Q$.  Since $Q$ is a near quantale, one has $y \leq x^\star$ if and only if $rxs \leq a$ implies $r y s \leq a$ for all $r,s \in Q$.  Therefore by Lemma \ref{vlemma} one has $\star = v(a)$.
\end{proof}

\begin{corollary}\label{simplequantales}
A unital near quantale $Q$ is simple if and only if, for any $x,y,a \in Q$ with $a < \bigvee Q$,
one has $x \leq y$ if and only if $rys \leq a$ implies $rxs \leq a$ for all $r,s \in Q$.
\end{corollary}

\begin{corollary}
A multiplicative lattice $Q$ is simple if and only if $Q = \{0,1\}$.
\end{corollary}

\begin{proof}
We may assume $0 < 1$.  Note that $r1s \leq 0$ implies $rxs = rsx \leq 0$ for all $x,r,s \in Q$, whence by Corollary \ref{simplequantales} one has $x \leq 1$ for all $x \in Q$.  Therefore, for any $a \in Q$, one has $(a \vee x)(a \vee y) = a^2 \vee ay \vee xa \vee xy \leq a \vee xy$ for all $x,y \in Q$.  It follows that the map $x \longmapsto a \vee x$ is a nucleus on $Q$, whence $a \vee x = x$ for all $x$ if $a < 1$. It follows, then, that $Q = \{0,1\}$.
\end{proof}


Proposition \ref{vquantales} and Corollary \ref{simplequantales} may be generalized to any near prequantale as follows.  Let $M$ be any magma.  For any $r \in M$, define self-maps $L_r$ and $R_r$ of $M$ by $L_r(x) = rx$ and $R_r(x) = xr$ for all $x \in M$, which we call {\it translations}.  Let $\Lin(M)$ denote the submonoid of the monoid of all self-maps of $M$ generated by the translations.  If $M$ is a monoid then any element of $\Lin(M)$ can be written in the form $L_r \circ R_s = R_s \circ L_r$ for some $r,s \in M$.  The proofs above readily generalize to yield the following.

\begin{proposition}\label{vprequantales}
Let $Q$ be a near prequantale and $a \in Q$.  Then one has
$$x^{v(a)} = \bigvee\{y \in Q: \ \forall f \in \Lin(Q) \ (f(x) \leq a \Rightarrow f(y) \leq a)\}$$
for all $x \in Q$; alternatively, $x^{v(a)}$ is the largest element of $Q$ such
that $f(x) \leq a$ implies $f(x^{v(a)}) \leq a$ for all $f \in \Lin(Q)$.
\end{proposition}

\begin{corollary}
A near prequantale $Q$ is simple if and only if, for any $x,y,a \in Q$ with $a < \bigvee Q$,
one has $x \leq y$ if and only if $f(y) \leq a$ implies $f(x) \leq a$ for all $f \in \Lin(Q)$.
\end{corollary}

The following notion of cyclic elements generalizes \cite[Definition 3.3.2]{ros}.  Let $M$ be an ordered magma.  We say that $a \in M$ is {\it cyclic} if $xy \leq a$ implies $yx \leq a$ for all $x,y \in M$.   If $M$ is associative, then $a$ is cyclic if and only if $x_1 x_2 x_3 \cdots x_n \leq a$ implies $x_2 x_3 \cdots x_n x_1 \leq a$ for any positive integer $n$ and any $x_1, x_2, \ldots, x_n \in M$.  If $M$ is near residuated, then we let $M_a^L$ (resp., $M_a^R$) denote the set of all $x \in M$ for which $a/x$ (resp., $x \backslash a$) is defined and we let $M_a^{LR} = M_a^L \cap M_a^R$.   In that case $a$ is cyclic if and only if $M_a^L = M_a^R$ and $a/x = x\backslash a$ for all $x \in M_a^{LR}$.  The following result generalizes \cite[Lemma 3.35]{gal} and the fact that $I^{v(D)} = (I^{-1})^{-1}$ for any integral domain $D$ with quotient field $F$ and any $I \in \F(D)$, where $I^{-1} = (D:_F I)$.

\begin{proposition}\label{vclosureprop3}
Let $M$ be a residuated ordered monoid, or a near residuated ordered monoid such that $\bigvee M$ exists.  Let $s$ be any element of $M$ such that $s = \bigvee M$ if $\bigvee M$ exists.  Then $v(a)$ exists for any cyclic element $a$ of $M$ and is given by $$x^{v(a)} = \left\{\begin{array}{ll} a/(a/x) =  \bigvee \{y \in M_a^{LR}: a/x \leq a/y\} & \mbox{if } x \in M_a^{LR} \\ s & \mbox{otherwise}\end{array}\right.$$ for all $x \in M$.
\end{proposition}

\begin{proof}
For any $x \in M$, let $x^\star = a/(a/x)$ if $x \in M_a^{LR}$ and $x^\star = \bigvee M$ otherwise.  The identity $a/(a/x) =  \bigvee \{y \in M_a^{LR}: a/x \leq a/y\}$ for $x \in M_a^{LR}$ is easy to check. We claim that $x^\star$ is the largest element $y$ of $M$ such that $rxs \leq a$ implies $rys \leq a$ for all $r,s \in M$.  By Lemma \ref{vlemma} the proposition follows from this claim.  First, note that $rxs \leq a$ implies $srx \leq a$, which implies that  $x \in M_a^{LR}$.  Therefore, if $x \notin M_a^{LR}$, then $\bigvee M$ is the largest element $y$ of $M$ described above.  Suppose, on the other hand, that $x \in M_a^{LR}$.  Then
$rxs \leq a \Rightarrow srx \leq a \Rightarrow sr \leq a/x \Rightarrow x^\star sr = (a/(a/x))sr  \leq a  \Rightarrow rx^\star s \leq a$.  Suppose that $y'$ is any element of $M$ such that $rxs \leq a \Rightarrow ry's \leq a$ for all $r,s \in M$.  Then since $(a/x)x1 \leq a$ we have $(a/x)y' \leq a$, whence $y'(a/x) \leq a$ and thus $y' \leq a/(a/x) = x^\star$.  This proves our claim.
\end{proof}

\begin{corollary}\label{simplenearmult}
A near multiplicative lattice $Q$ is simple if and only if $x/y$ exists and $x/(x/y) = y$ for all $x,y \in Q$ such that $x,y < \bigvee Q$.  In particular, a multiplicative lattice $Q$ is simple if and only if $x/(x/y) = y$ for all $x,y < \bigvee Q$.
\end{corollary}

\begin{example}
Let $G$ be a bounded complete partially ordered abelian group.  
Let $G[\infty] = G \amalg \{\infty\}$, where $a \infty = \infty a = \infty$ and $a < \infty$ for all $a \in G$.  Then $G[\infty]$ is a simple near multiplicative lattice, by Corollary \ref{simplenearmult}.  However, the multiplicative lattice
$G[\pm \infty] = G[\infty] \amalg \{-\infty\}$, where $-\infty = \bigwedge G$ is an annihilator of $G[\infty]$, is not simple.  For example, if $D$ is a DVR with quotient field $F$, then $\F(D) \cong \ZZ[\infty]$ is simple, but $\operatorname{Mod}_D(F) \cong \ZZ[\pm \infty]$ is not simple.
\end{example}


Finally we generalize the well-known fact that $I^{v(D)} = \bigcap \{xD : x \in F \mbox{ and } xD \supseteq I\}$ for any integral domain $D$ with quotient field $F$ and any $I \in \F(D)$.

\begin{proposition}\label{semiu}
Let $Q$ be an associative unital semi-$\U$-lattice, and let $a \in Q$.  Then $v(a)$ exists and $x^{v(a)} = \bigwedge\{uav: \ u,v \in \U(Q) \mbox{ and } x \leq uav\}$ for all $x \in Q$.
\end{proposition}

\begin{proof}
For all $x \in Q$, let $x^\star = \bigwedge\{uav: \ u,v \in \U(Q), \ x \leq uav\}$, which exists since $Q$ is bounded complete.  Clearly $\star$ is a closure operation on $Q$ with $a^\star = a$.  Moreover, for all $x \in Q$ and $w \in \U(Q)$ one has
$wx^\star  =  \bigwedge\{wuav:  \ u,v \in \U(Q), \ wx \leq wuav\} 
	 =  \bigwedge\{u'av:  \ u',v \in \U(Q), \ wx \leq u'av\} 
	 =  (wx)^\star,$
and likewise $x^\star w = (xw)^\star$.  Thus $\star$ is a nucleus on $Q$ by Proposition \ref{closureprop2}.  Now let $\star' \in \N(Q)$ with $a^{\star'} = a$, and let $x \in Q$.  If $x \leq uav$ for $u,v \in \U(Q)$, then $x^{\star'} \leq ua^{\star'}v = uav$, whence $x^{\star'} \leq x^\star$.  Therefore $\star' \leq \star$.  It follows that $v(a)$ exists and equals $\star$.
\end{proof}




\section{Finitary nuclei and precoherence}\label{sec:PFN}


A closure operation $\star: S \longrightarrow S$ on a poset $S$ may not be Scott continuous, even though the corestriction ${_{S^\star}}|\star: S \longrightarrow S^\star$ of $\star$ to $S^\star$ is always sup-preserving (hence Scott continuous).  If $\star: S \longrightarrow S$ is Scott continuous then we will say that $\star$ is {\it finitary}.  (Such closure operations are often said to be {\it algebraic}.)  Equivalently this means that $(\bigvee \Delta)^\star = \bigvee(\Delta^\star)$, or $\bigvee(\Delta^\star) \in S^\star$, for any directed subset $\Delta$ of $S$ such that $\bigvee \Delta$ exists.  Finitary closure operations generalize finite type star and semistar operations, ideal systems, and module systems.

For any poset $S$, let $\C_f(S)$ denote the poset of all finitary closure operations on $S$, and for any ordered magma $M$, let $\N_f(M) = \N(M) \cap \C_f(M)$ denote the poset of all finitary nuclei on $M$.  In this section we study properties of the poset $\N_f(M)$.  A summary of these results and the results for $\N(M)$ from the previous section is provided in Table 2.

\begin{table} 
\caption{The posets $\N(M)$ and $\N_f(M)$}
\centering  \ \\
\begin{tabular}{l|l|l} 
If $M$ is $\ldots$ & $\N(M)$ is $\ldots$ & $\N_f(M)$ is $\ldots$  \\ \hline\hline
near sup-magma & complete & \\ \hline
bounded complete & bounded complete  &  \\ \hline
meet semilattice & meet semilattice & \\ \hline
Scott-topological dcpo magma & & complete \\ \hline
Scott-topological bdcpo magma & & bounded complete \\ \hline
Scott-topological near sup-magma & complete & complete \\ \hline
Scott-topological, bounded complete & bounded complete & bounded complete  \\ \hline
algebraic meet semilattice &  meet semilattice & meet semilattice
\end{tabular}
\end{table}

Recall that an element $x$ of a poset $S$ is said to be {\it compact} if, whenever $x \leq \bigvee \Delta$ for some directed subset $\Delta$ of $S$ such that $\bigvee \Delta$ exists, one has $x \leq y$ for some $y \in \Delta$.  Any minimal element of $S$, for example, is compact.  We let $\K(S)$ denote the set of all compact elements of $S$.

\begin{example} \
\begin{enumerate}
\item For any set $X$, the set $\K(2^X)$ of compact elements of the complete lattice $2^X$ is equal to the set of all finite subsets of $X$.
\item  For any algebra $A$ over a ring $R$, the compact elements of the quantale $\operatorname{Mod}_R(A)$ of all $R$-submodules of $A$ are precisely the finitely generated $R$-submodules of $A$.
\end{enumerate}
\end{example}

A element $x$ of a poset $S$ is {\it algebraic} if $x$ is the supremum of a set of compact elements of $S$, or equivalently, if one has $x = \bigvee \{y \in \K(S): \ y \leq x\}$.   A poset is said to be {\it algebraic} if all of its elements are algebraic.  Note, however, that an algebraic complete lattice is said to be an {\it algebraic lattice}.  For example, for any set $X$, the poset $2^X$ is an algebraic lattice, and the poset $2^X \setmin \{\emptyset\}$ is an algebraic near sup-lattice with $\K(2^X \setmin \{\emptyset\}) = \K(2^X) \setmin \{\emptyset\}$.

For any set $X$ and any set $\Gamma$ of self-maps of $X$, let $\langle \Gamma \rangle$ denote the submonoid generated by $\Gamma$ of the monoid of all self-maps of $X$ under composition.

\begin{lemma}
Let $S$ be a poset.
\begin{enumerate}
\item If $S$ is a dcpo (resp., bdcpo) then $\C_f(S)$ is complete (resp., bounded complete) and one has $\bigvee_{\C_f(S)} \Gamma = \bigvee_{\C(S)} \Gamma$ and $x^{\bigvee_{\C(S)} \Gamma} = \bigvee\{x^\gamma: \ \gamma \in \langle \Gamma \rangle\}$ for all $x \in S$ and for any subset (resp., any bounded subset) $\Gamma$ of $\C_f(S)$. 
\item If $S$ is a near sup-lattice, then $\C_f(S)$ is a sub sup-lattice of the sup-lattice $\C(S)$.
\item If $S$ is an algebraic meet semilattice, then $\C_f(S)$ is a sub meet semilattice of the meet semilattice $\C(S)$.
\end{enumerate}
\end{lemma}


\begin{proof} \
\begin{enumerate}
\item Define a preclosure $\sigma$ of $S$ by $x^\sigma = \bigvee\{x^\gamma: \ \gamma \in \langle \Gamma \rangle\}$ for all $x \in S$.  The map $\sigma$ is defined because the set $\{x^\gamma: \ \gamma \in \langle \Gamma \rangle\}$ is directed (and bounded above if $\Gamma$ is bounded above) for any  $x \in S$.  Let $x \in S$, and let $X = \{x^\gamma: \ \gamma \in \langle \Gamma \rangle\}$.  Then $(\bigvee X)^\star = \bigvee\{(x^\gamma)^\star: \ \gamma \in \langle \Gamma \rangle\} \leq \bigvee X$ for all $\star \in \Gamma$, whence $(\bigvee X)^\gamma = \bigvee X$, and therefore $(x^\sigma)^\gamma = x^\sigma$, for all $\gamma \in \langle \Gamma \rangle$.  Therefore $(x^\sigma)^\sigma = x^\sigma$, so $\sigma$ is a closure operation on $S$, and clearly $\sigma =\bigvee_{\C(S)} \Gamma$.  It remains only to show that $\sigma$ is finitary, since it will follow that $\sigma = \bigvee_{\C_f(S)} 
\Gamma$.  Let $\Delta$ be a directed subset of $S$.  Then we have
$\left(\bigvee \Delta\right)^\sigma = \bigvee\left\{\left(\bigvee \Delta\right)^\gamma: \ \gamma \in \langle \Gamma \rangle\right\}
  = \bigvee\left\{\bigvee (\Delta^\gamma): \ \gamma \in \langle \Gamma \rangle\right\}  
  = \bigvee \{x^\gamma: \ x \in \Delta, \gamma \in \langle \Gamma \rangle\} 
  = \bigvee (\Delta^\sigma).$
Thus $\sigma$ is finitary.
\item This follows from (1).
\item Let $\star, \star' \in \C_f(S)$.  By Lemma \ref{closelem} it suffices to show that $\star \wedge \star' \in \C(S)$ is finitary.  Let $\Delta$ be a directed subset of $S$ such that $x = \bigvee \Delta$ exists, and let $z = x^{\star \wedge \star'}$.  Note first that $z$ is an upper bound of $\Delta^{\star \wedge \star'}$.  Let $u$ be any upper bound of $\Delta^{\star \wedge \star'}$.  We claim that $u \geq z$.   To show this, let $t$ be any compact element of $S$ with $t \leq z = x^\star \wedge x^{\star'}$.   Then $t \leq x^\star = \bigvee(\Delta^\star)$ and $t \leq x^{\star'} = \bigvee(\Delta^{\star'})$, whence $t \leq y^\star$ and $t \leq (y')^{\star'}$ for some $y,y' \in \Delta$.  Since $\Delta$ is directed, we may choose $w \in \Delta$ with $y,y' \leq w$.  Then $t \leq w^\star \wedge w^{\star'} = w^{\star \wedge \star'} \leq u$.  Therefore, taking the supremum over all compact $t \leq z$, we see that $z \leq u$.  Therefore $(\bigvee \Delta)^{\star \wedge \star'} = z = \bigvee (\Delta^{\star \wedge \star'})$.  Thus $\star \wedge \star'$ is finitary.
\end{enumerate}
\end{proof}

\begin{proposition}\label{Nf}
Let $M$ be an ordered magma.
\begin{enumerate}
\item If $M$ is a Scott-topological dcpo magma (resp., Scott-topological bdcpo magma) then $\N_f(M)$ is complete (resp., bounded complete) and one has $\bigvee_{\N_f(M)} \Gamma = \bigvee_{\N(M)} \Gamma$ and $x^{\bigvee_{\N(M)} \Gamma} = \bigvee\{x^\gamma: \ \gamma \in \langle \Gamma \rangle\}$ for all $x \in M$ and for any subset (resp., any bounded subset) $\Gamma$ of $\N_f(M)$. 
\item If $M$ is a Scott-topological near sup-magma, then $\N_f(M)$ is a sub sup-lattice of the sup-lattice $\N(M)$.
\item If $M$ is an algebraic meet semilattice, then $\N_f(M)$ is a sub meet semilattice of the meet semilattice $\N(M)$.
\end{enumerate}
\end{proposition}

\begin{proof}
By the lemma it suffices to observe that the closure operation $x \longmapsto \bigvee\{x^\gamma: \ \gamma \in \langle \Gamma \rangle\}$, when defined for $\Gamma \subseteq \C_f(M)$, is a nucleus if $M$ is Scott-topological.
\end{proof}

The remainder of this section is devoted to proving Theorem \ref{precoherentthm} of the introduction.

\begin{lemma}\label{joinlemma}
Let $S$ be a join semilattice.  Then $x \in S$ is compact if and only if, whenever $x \leq \bigvee X$ for some subset $X$ of $S$ such that $\bigvee X$ exists, one has $x \leq \bigvee Y$ for some finite subset $Y$ of $X$.  Moreover, $\K(S)$ is closed under finite suprema.  
\end{lemma}


\begin{lemma}\label{finitetypeprop}
Let $\star$ be a closure operation on a poset $S$.
\begin{enumerate}
\item If $x^\star = x \vee \bigvee\{y^\star: \ y \in \K(S), \ y \leq x\}$ for all $x \in S$, then $\star$ is finitary.
\item If $\star$ is finitary and $S$ is an algebraic join semilattice, then $x^\star = \bigvee\{y^\star: \ y \in \K(S), \ y \leq x\}$ for all $x \in S$, and one has $\K(S^\star) \subseteq \K(S)^\star$.
\item If $\star$ is finitary and $S$ is a bdcpo, then $\K(S)^\star \subseteq \K(S^\star)$.
\item If $\star$ is finitary, and if $S$ is an algebraic bounded complete join semilattice (resp., algebraic near sup-lattice, algebraic lattice), then so is $S^\star$, and $\K(S^\star) = \K(S)^\star$.
\end{enumerate}
\end{lemma}

\begin{proof} \
\begin{enumerate}
\item Let $\Delta$  be a directed subset of $S$ such that $x = \bigvee \Delta$ exists.   We wish to show that $\bigvee (\Delta^\star)$ exists and equals $x^\star$.  Clearly $x^\star$ is an upper bound of $\Delta^\star$, and if $w$ is any upper bound of $\Delta^\star$, then we claim that $x^\star \leq w$.  Let $y \in \K(S)$ with $y \leq x$.  Then $y \leq z$ for some $z \in \Delta$.  Therefore $y^\star \leq z^\star$, where $z^\star \in \Delta^\star$.   Thus we have $y^\star \leq w$.  Taking the supremum over all such $y$, we see that $\bigvee\{y^\star: \ y \in \K(S), \ y \leq x\} \leq w$.  Moreover, we have $x = \bigvee \Delta \leq w$, and therefore $x^\star = x \vee \bigvee\{y^\star: \ y \in \K(S), \  y \leq x\} \leq w$, as claimed.
\item Let $x \in S$.  One has $x = \bigvee \Delta$, where $\Delta =  \{y \in \K(S): \ y \leq x\}$ is directed by Lemma \ref{joinlemma}.  Therefore $x^\star = \bigvee (\Delta^\star)$.   Suppose  that $x \in \K(S^\star)$.  Then $x = \bigvee(\Delta^\star)$ and $\Delta^\star \subseteq S^\star$ is directed, so $x \leq y^\star$ for some $y \in \Delta$, whence $x = y^\star$.  Therefore $\K(S^\star) \subseteq \K(S)^\star$.
\item Let $x \in \K(S)$.  Suppose that $x^\star \leq \bigvee_{S^\star}\Delta$ for some directed subset $\Delta$ of $S^\star$ such that $\bigvee_{S^\star}\Delta$ exists.  Then $\Delta$ is directed and bounded above, whence $\bigvee_S \Delta$ exists.  It follows that $\bigvee_{S^\star}\Delta = (\bigvee_S \Delta)^\star = \bigvee_S (\Delta^\star) = \bigvee_S \Delta$.  Therefore $x \leq x^\star \leq \bigvee_S \Delta$, so $x \leq y$ for some $y \in \Delta$, whence $x^\star \leq y^\star = y$.  Thus we have $x^\star \in \K(S^\star)$.
\item Since $\K(S)^\star \subseteq \K(S^\star)$ by (3), it follows from (2) that $S^\star$ is algebraic and $\K(S)^\star = \K(S^\star)$.  Since $S^\star$ is bounded complete (resp., near sup-complete, complete) by Lemma \ref{starlemma}(4), it follows that $S^\star$ is an algebraic bounded complete join semilattice (resp., algebraic near sup-lattice, algebraic lattice). 
\end{enumerate}
\end{proof}

For any closure operation $\star$ on a bounded complete poset $S$, let $\star_f$ denote the operation on $S$ defined by $x^{\star_f} = \bigvee\{y^\star: \ y \in \K(S) \mbox{ and } y \leq x\}$ for all $x \in S$.

\begin{proposition}\label{finitaryclosure}
Let $\star$ be a closure operation on an algebraic bounded complete join semilattice $S$.  Then $\star_f$ is the largest finitary closure operation on $S$ that is smaller than $\star$, one has $x^{\star_f} = x^\star$ for all $x \in \K(S)$, and $\K(S^{\star_f}) = \K(S)^{\star_f}$.
\end{proposition}

\begin{proof}
Clearly $\star_f$ is a preclosure on $S$.  Let $x \in S$.  The set $\{y^\star: \ y \in \K(S), \ y \leq x\}$ is directed by Lemma \ref{joinlemma}.  Therefore if $z \in \K(S)$ and $z \leq x^{\star_f}$ then $z \leq y^\star$ for some $y \in \K(S)$ such that $y \leq x$, in which case $z^\star \leq y^\star$.  Therefore $(x^{\star_f})^{\star_f} =  \bigvee \{z^\star: \ z \in \K(S), \ z \leq x^{\star_f}\} \leq \bigvee \{y^\star: \ y \in \K(S), \ y \leq x\} = x^{\star_f}.$
Thus $\star$ is a closure operation on $S$, and the rest of the proposition follows from Lemma \ref{finitetypeprop}.
\end{proof}

Generalizing \cite[Definition 4.1.1]{ros}, we say that an ordered magma $M$ is {\it precoherent} if $M$ is algebraic and $\K(M)$ is closed under multiplication, and we say that $M$ is {\it coherent} if $M$ is precoherent and unital with $1$ compact.

\begin{example} \
\begin{enumerate}
\item For any magma $M$, the prequantale $2^M$ is precoherent, and if $M$ is unital then $2^M$ is coherent.
\item A {\it $\K$-lattice} is a precoherent multiplicative lattice.  We say that a {\it near $\K$-lattice} is a precoherent near multiplicative lattice.  For example, if $M$ is a commutative monoid, then $2^M$ is a $\K$-lattice and $2^M\setmin\{\emptyset\}$ is a near $\K$-lattice.
\item For any integral domain $D$ with quotient field $F$, the ordered monoid $\F(D)$ of all nonzero $D$-submodules of $F$ is a near $\K$-lattice.
\item Let $A$ be an algebra over a ring $R$.  The quantale $\operatorname{Mod}_R(A)$ of all $R$-submodules of $A$ is precoherent, and it is coherent if and only if $A$ is finitely generated as an $R$-module.  In particular, if $R$ and $A$ are commutative then $\operatorname{Mod}_R(A)$ is a $\K$-lattice.
\end{enumerate}
\end{example}


The following result generalizes the corresponding well-known facts for star and semistar operations and ideal and module systems.

\begin{theorem}\label{klattice}
If $\star$ is a nucleus on a precoherent semiprequantale (resp., precoherent near prequantale, precoherent prequantale) $Q$, then $\star_f$ is the largest finitary nucleus on $Q$ that is smaller than $\star$, the ordered magma $Q^{\star_f}$ is also a precoherent semiprequantale (resp., precoherent near prequantale, precoherent prequantale), and $\K(Q^{\star_f}) = \K(Q)^{\star_f}$.
\end{theorem}

\begin{proof}
By Proposition \ref{finitaryclosure}, to prove the first claim we need only show that $\star_f$ is a nucleus.  Let $x \in Q$ and $a \in \K(Q)$.  We have $(ax)^{\star_f} = \bigvee \{y^\star: \ y \in \K(Q), \ y \leq ax\}$ and $ax^{\star_f} = \bigvee \{az^\star: \ z \in \K(Q), \ z \leq x\}$.  Let $z \in \K(Q)$ with $z \leq x$, and let $y = az$.  By our hypotheses on $Q$ we have $y \in \K(Q)$, and $y = az \leq ax$.  Therefore $az^\star \leq y^\star 
\leq (ax)^{\star_f} $.  Taking the supremum over all such $z \in \K(Q)$ we see that $ax^{\star_f} \leq (ax)^{\star_f}$.  By symmetry we also have $x^{\star_f}a \leq (xa)^{\star_f}$.  By Proposition \ref{closureprop2}, then, it follows that $\star_f$ is a nucleus.

To prove the second claim we may assume $\star = \star_f$ is finitary.  Then $Q^\star$ is a an algebraic semiprequantale (resp., algebraic near prequantale, algebraic prequantale) by Lemma \ref{finitetypeprop}(3) and Proposition \ref{CSTstar}.  Let $x,y \in \K(Q^\star)$, and suppose that $x \star y \leq \bigvee_{Q^\star} \Delta$ for some directed subset $\Delta$ of $Q^\star$ such that $\bigvee_{Q^\star} \Delta$ exists.  Then $\bigvee_Q \Delta$ exists, and $xy \leq \bigvee_{Q^\star} \Delta = (\bigvee_Q \Delta)^\star = \bigvee_Q \Delta$.  Therefore, since $xy \in \K(Q)$, we have $xy \leq z$, whence $x \star y \leq z$, for some $z \in \Delta$.  Thus $x \star y \in \K(Q^\star)$  
and $Q^\star$ is precoherent.  Finally, the third claim follows from Proposition \ref{finitaryclosure}.
\end{proof}




\begin{corollary}\label{vclospropcor}
Let $\star$ be a nucleus on a precoherent near prequantale $Q$.   Then the nucleus $t(a) = v(a)_f$ for any $a \in Q$ is the largest finitary nucleus $\star$ on $Q$ such that $a^\star = a$, and one has $\star_f = \bigwedge\{t(a): \ a \in Q^{\star_f}\}.$  More generally,  the nucleus $t(S) =  v(S)_f$ for any subset $S$ of $Q$ is the largest finitary nucleus $\star$ on $Q$ such that $S \subseteq Q^\star$, and one has $\star_f = t(Q^{\star_f})$.
\end{corollary}


The following result reveals a few subclasses of the precoherent semiprequantales.

\begin{proposition}\label{ucoherent}
Let $M$ be an ordered monoid.  Then $M$ is a $\U$-lattice (resp., near $\U$-lattice, semi-$\U$-lattice) with $1$ compact if and only if $M$ is a coherent prequantale (resp., coherent near prequantale,  coherent semiprequantale) such that every compact element of $M$ is the supremum of a subset of $\U(M)$; in that case the compact elements of $M$ are precisely the suprema of the finite subsets of $\U(M)$. 
\end{proposition}

Proposition \ref{ucoherent} follows easily from the following lemma, which generalizes the fact that every invertible fractional ideal of an integral domain is finitely generated.

\begin{lemma}\label{1compact}
The following are equivalent for any unital ordered magma $M$.  
\begin{enumerate}
\item $1 \in \K(M)$.
\item $\U(M) \subseteq \K(M)$.
\item $\U(M) \cap \K(M) \neq \emptyset$.
\end{enumerate}
\end{lemma}


In the remainder of this section we apply the theory of nuclei to construct a representation of any precoherent near prequantale (resp., precoherent prequantale) as the ideal completion of some multiplicative semilattice (resp., prequantic semilattice), and conversely a representation of any multiplicative semilattice (resp., prequantic semilattice) as the ordered magma of the compact elements of some precoherent near prequantale (resp., precoherent prequantale).  These representations yield appropriate category equivalences.

A subset $I$ of a poset $S$ is said to be an {\it ideal} of $S$ if
$I$ is a directed downward closed subset of $S$.  For any $x \in S$, the set
$$\downarrow\!\! x= \{y \in S: \ y \leq x\}$$ is an ideal of $S$ called  the {\it principal ideal generated by $x$}.  The {\it ideal completion} $\Idl(S)$ of $S$ is the set of all ideals of $S$ partially ordered by the subset relation.  
If $S$ is a join semilattice, then, for any nonempty subset $X$ of $S$, we let $$\downarrow \!\! X = \{y \in S: \ y \leq \bigvee T \mbox{ for some finite nonempty } T \subseteq X\}.$$

\begin{lemma}\label{downarrowlemma}
Let $S$ be a join semilattice.
\begin{enumerate}
\item For any nonempty subset $X$ of $S$, the set $\downarrow\!\! X$ is the smallest ideal of $S$ containing $X$.
\item The operation $\downarrow$ is a finitary closure operation on the algebraic near sup-lattice $2^S \setmin \{\emptyset\}$ with $\operatorname{im}\!\! \downarrow \ = \Idl(S)$. 
\item If $S$ has a least element, then the operation $\downarrow$ extends uniquely to a finitary closure operation on the algebraic sup-lattice $2^S$ such that $\downarrow\!\! \emptyset = \bigwedge S$.
\end{enumerate}
\end{lemma}

\begin{proof}
It is straightforward to show that $\downarrow \!\! X$ is the smallest ideal of $S$ containing $X$ and $\downarrow$ is a closure operation on $2^S \setmin\{\emptyset\}$.  To show that $\downarrow$ is finitary, we show that $\downarrow \!\! X = \bigcup\downarrow \!\! \{T:\ T \subseteq X \mbox{ finite and nonempty}\}$ for any $X \in 2^S \setmin\{\emptyset\}$.  Let $x \in \downarrow \!\! X$.  Then $x \leq \bigvee T$, where $T \subseteq X$ is finite and nonempty. But $\bigvee T \in \downarrow \!\! T$, so $x \in \downarrow \!\! T$.  This proves (1) and (2), and (3) is clear.
\end{proof}

The following result generalizes \cite[Proposition I-4.10]{gie}.

\begin{proposition}\label{algebraicprop2}
Let $L$ be an algebraic near sup-lattice and $S$ a join semilattice.
\begin{enumerate}
\item $\Idl(S)$ is an algebraic near sup-lattice.
\item $\K(L)$ is a sub join semilattice of $L$.
\item The map $S \longrightarrow \K(\Idl(S))$ acting by $x \longmapsto \downarrow\!\! x$ is a poset isomorphism.
\item The map $\Idl(\K(L)) \longrightarrow L$ acting by $I \longmapsto \bigvee I$ is a poset isomorphism with inverse acting by $x \longmapsto (\downarrow \!\! x) \cap \K(L)$.
\end{enumerate}
\end{proposition}

\begin{proof}
By Lemma \ref{downarrowlemma} the map $\downarrow$ is a finitary nucleus on the algebraic near sup-lattice $2^S \setmin \{\emptyset \}$ with $\operatorname{im} \downarrow \ = \Idl(S)$.  By Lemma \ref{finitetypeprop}, then, it follows that $\Idl(S)$ is an algebraic near sup-lattice with $\K(\Idl(S)) = \downarrow \!\!(\K(2^S\setmin \{\emptyset\})) = \{\downarrow\!\! x: x \in S\}$.  This implies (1) and (3).  Next, $\K(L)$ is nonempty and therefore a sub join semilattice of $L$ by Lemma \ref{joinlemma}.  This proves (2).  Finally, the map $f: I \longmapsto \bigvee I$ in (4) is well-defined because $L$ is a near sup-lattice, the map $f$ is clearly order-preserving, and the map $x \longmapsto (\downarrow\!\! x) \cap \K(L)$ is an order-preserving inverse to $f$ because $L$ is algebraic and $\K(L)$ is a join semilattice.
\end{proof}

Theorems \ref{maintheorem} and \ref{maintheorem2} below generalize \cite[Proposition 4.1.4]{ros} and are analogues of Proposition \ref{algebraicprop2} for precoherent near prequantales and precoherent prequantales, respectively.  First, we show that the ideal completion of a multiplicative semilattice (resp., prequantic semilattice) is a precoherent near prequantale (resp., precoherent prequantale) under the operation of {\it $\downarrow$-multiplication}, defined by  $(I,J) \longmapsto \downarrow\!\! (IJ)$.

\begin{lemma}\label{replemma} One has the following.
\begin{enumerate}
\item Let $M$ be a multiplicative semilattice.  The operation $\downarrow: X \longmapsto \downarrow\!\! X$ is a finitary nucleus on the precoherent near prequantale $2^M \setmin\{\emptyset\}$, one has $\Idl(M) = \downarrow\!\! (2^M \setmin\{\emptyset\})$, and $\Idl(M)$ is a precoherent near prequantale under $\downarrow$-multiplication.
\item Let $M$ be a prequantic semilattice.  The operation $\downarrow: X \longmapsto \downarrow\!\! X$ is a finitary nucleus on the precoherent prequantale $2^M$, one has $\Idl(M) = \downarrow\!\! (2^M)$, and $\Idl(M)$ is a precoherent prequantale under $\downarrow$-multiplication.
\end{enumerate}
\end{lemma}

\begin{proof}
We prove statement (1).  The proof of (2) is similar.  Let $X,Y \in 2^M \setmin\{\emptyset\}$, and let $z \in \downarrow\!\! X$ and $w \in \downarrow\!\! Y$.  Then $z \leq \bigvee S$ and $w \leq \bigvee T$ for some finite nonempty sets $S \subseteq X$ and $T \subseteq Y$.  Therefore $zw \leq \bigvee S \bigvee T = \bigvee (ST)$, where $ST$ is a finite nonempty subset of $XY$, so $zw \in \downarrow\!\! (XY)$.  Thus we have $\downarrow\!\! X \downarrow\!\! Y \subseteq \downarrow\!\! (XY)$, so $\downarrow$ is a finitary nucleus on $2^M \setmin\{\emptyset\}$ by Lemma \ref{downarrowlemma}.  By Theorem \ref{klattice}, then, $\Idl(M)$ is a precoherent near prequantale.
\end{proof}

The morphisms in the {\it category of precoherent near prequantales} (resp., {\it category of precoherent prequantales}) are morphisms $f: Q \longrightarrow Q'$ of near prequantales (resp., prequantales), with $Q$ and $Q'$ precoherent, such that $f(\K(Q)) \subseteq \K(Q')$.

\begin{theorem}\label{maintheorem}
Let $Q$ be a precoherent near prequantale and let $M$ be a multiplicative semilattice.
\begin{enumerate}
\item $\Idl(M)$ is a precoherent near prequantale under $\downarrow$-multiplication.
\item $\K(Q)$ is a multiplicative semilattice.
\item The map $M \longrightarrow \K(\Idl(M))$ acting by $x \longmapsto \downarrow\!\! x$ is an isomorphism of ordered magmas.
\item The map $\Idl(\K(Q)) \longrightarrow Q$ acting by $I \longmapsto \bigvee I$ is an isomorphism of ordered magmas.
\item If $f: Q \longrightarrow Q'$ is a morphism of precoherent near prequantales, then the map $\K(f): \K(Q) \longrightarrow \K(Q')$ given by $\K(f)(x) = f(x)$ for all $x \in \K(Q)$ is a morphism of multiplicative semilattices.
\item If $g: M \longrightarrow M'$ is a morphism of multiplicative semilattices, then the map $\Idl(g): \Idl(M) \longrightarrow \Idl(M')$ given by $\Idl(g)(I) = \downarrow\!\! (g(I))$ for all $I \in \Idl(M)$ is a morphism of precoherent near prequantales.
\item The associations $\K$ and $\Idl$ are functorial and provide an equivalence of categories between the category of precoherent near prequantales and the category of multiplicative semilattices.
\end{enumerate}
\end{theorem}

\begin{proof}
Statement (1) follows from Lemma \ref{replemma}, and (2) is clear.  To prove (3), note that the map $M \longrightarrow \K(\Idl(M)) \subseteq \Idl(M)$ is the composition $M \longrightarrow  2^M\setmin\{\emptyset\} \longrightarrow \ \downarrow\!\! (2^M \setmin \{\emptyset\}) = \Idl(M)$ of magma homomorphisms and is therefore a magma homomorphism.  Thus it is an isomorphism of ordered magmas, by Proposition \ref{algebraicprop2}(3).  The map in statement (4) is a magma homomorphism by Proposition \ref{nearprequantales} and therefore is an isomorphism of ordered magmas by Proposition \ref{algebraicprop2}(4).  Statement (5) is clear.  To prove (6), first note that $g(\downarrow\!\! X) \subseteq \ \downarrow\!\! (g(X))$ for any nonempty subset $X$ of $M$.   The map $\Idl(g)$ is then easily shown to be a morphism of near prequantales, and since $\Idl(g)(\downarrow\!\! x) = \downarrow\!\! (g(x))$ for all $x \in Q$, it follows from (3) that $\Idl(g)$ is a morphism of precoherent near prequantales.  Finally, (7) follows from (1) through (6).
\end{proof}

\begin{corollary} The following are equivalent for any ordered magma $M$.
\begin{enumerate}
\item $M$ is a precoherent near prequantale.
\item $M$ is isomorphic to the ideal completion $\Idl(N)$ under $\downarrow$-multiplication of some multiplicative semilattice $N$.
\item $\K(M)$ is a sub multiplicative semilattice of $M$ and $M$ is isomorphic to the ordered magma $\Idl(\K(M))$ under $\downarrow$-multiplication.
\item $M$ is isomorphic to $(2^N \setmin\{\emptyset\})^\star$ for some multiplicative semilattice $N$ and some finitary nucleus $\star$ on $2^N \setmin\{\emptyset\}$.
\end{enumerate}
\end{corollary}

\begin{corollary} An ordered magma $M$ is a multiplicative semilattice if and only if $M$ is isomorphic to $\K(Q)$ for some precoherent near prequantale $Q$.
\end{corollary}

Similar proofs of the above results yield the following.

\begin{theorem}\label{maintheorem2}
One has the following.
\begin{enumerate}
\item If $f: Q \longrightarrow Q'$ is a morphism of precoherent prequantales, then the map
$\K(f): \K(Q) \longrightarrow \K(Q')$ given by $\K(f)(x) = f(x)$ for all $x \in \K(Q)$ is a morphism of prequantic semilattices.
\item If $g: M \longrightarrow M'$ is a morphism of prequantic semilattices, then the map $\Idl(g): \Idl(M) \longrightarrow \Idl(M')$ given by $\Idl(g)(I) = \downarrow\!\! (g(I))$ for all $I \in \Idl(M)$ is a morphism of precoherent prequantales.
\item The associations $\K$ and $\Idl$ are functorial and provide an equivalence of categories between the category of precoherent prequantales and the category of prequantic semilattices.
\end{enumerate}
\end{theorem}


\begin{corollary} The following are equivalent for any ordered magma $M$.
\begin{enumerate}
\item $M$ is a precoherent prequantale.
\item $M$ is isomorphic to the ideal completion $\Idl(N)$ under $\downarrow$-multiplication of some prequantic semilattice $N$.
\item $\K(M)$ is a sub prequantic semilattice of $M$ and $M$ is isomorphic to the ordered magma $\Idl(\K(M))$ under $\downarrow$-multiplication.
\item $M$ is isomorphic to $(2^N)^\star$ for some prequantic semilattice $N$ and some finitary nucleus $\star$ on $2^N$.
\end{enumerate}
\end{corollary}

\begin{corollary} An ordered magma $M$ is a prequantic semilattice if and only if $M$ is isomorphic to $\K(Q)$ for some precoherent prequantale $Q$.
\end{corollary}

The results above also specialize to coherent prequantales and unital prequantic semilattices, coherent near prequantales and unital multiplicative semilattices, and to the corresponding settings where all operations in question are associative and/or commutative.

\section{Stable nuclei}\label{sec:stable}

A semistar operation $\star$ on an integral domain $D$ is said to be {\it stable} if 
$(I\cap J)^\star = I^\star \cap J^\star$ for all $I,J \in \F(D)$.  If $\star$ is stable then $(I:_F J)^\star = (I^\star :_F J)$ for all $I,J \in \F(D)$ with $J$ finitely generated, where $F$ is the quotient field of $D$.  For any semistar operation $\star$ on $D$ there exists a largest stable semistar operation $\overline{\star}$ on $D$ that is smaller than $\star$, given by
$$I^{\overline{\star}} = \bigcup\{(I:_F J): \ J \subseteq D \mbox{ and } J^\star = D^\star\}$$
for all $I \in \F(D)$.  If $\star$ is of finite type then so is $\overline{\star}$, whence $\star_w = \overline{\star_f}$ is the largest stable semistar operation of finite type that is smaller than $\star$.  The {\it $w$ operation} $w = v_w = \overline{t}$ is the largest stable semistar operation on $D$ of finite type such that $D^w = D$.  In this section we generalize these results to the context of semimultiplicative lattices.   All of the results of this section, excluding only Proposition  \ref{stableU}, apply to any coherent residuated semimultiplicative lattice that is also a meet semilattice.  In particular, they apply to star, semistar, and semiprime operations and ideal and module systems.  Unfortunately, we do not know if the results of this section generalize to nonassociative, noncommutative, or nonunital semiprequantales.



We will say that a nucleus $\star$ on a near residuated ordered magma $M$ is {\it stable} if $(\bigwedge X)^\star = \bigwedge(X^\star)$ for any finite subset $X$ of $M$ such that $\bigwedge X$ exists and $(x/t)^\star = x^\star/t$ (resp., $(t \backslash x)^\star = t \backslash x^\star)$ for all $x,t \in M$ with $t$ compact such that $x/t$ (resp., $t \backslash x$) exists.  In general the first condition above does not necessarily imply the second.  (The first condition is automatic if $M$ is a locale.) However, by the following proposition, the implication holds if every compact element of $M$ is the supremum of a finite subset of $\Inv(M)$, which in turn holds for any associative unital semi-$\U$-lattice $M$, such as the near $\U$-lattice $M = \F(D)$ and the semi-$\U$-lattice $\Fp(D)$.

\begin{proposition}\label{stableU}
Let $M$ be  an associative unital semi-$\U$-lattice, or more generally a near residuated ordered magma such that every compact element of $M$ is the supremum of a finite subset of $\Inv(M)$.  A nucleus $\star$ on $M$ is stable if and only if $(\bigwedge X)^\star = \bigwedge(X^\star)$ for any finite subset $X$ of $M$ such that $\bigwedge X$ exists.  Moreover, if $M$ is also a meet semilattice, then $x/t$ and $t \backslash x$ exist in $M$ for all $x,t \in M$ with $t$ compact.
\end{proposition}

\begin{proof}
Suppose that $\star$ distributes over finite meets.  Let $x,t \in M$ with $t$ compact, and suppose that $x/t$ exists.  We may write $t = u_1 \vee u_2 \vee \cdots \vee u_n$ with each $u_i \in \Inv(M)$.  Since $x/u$ exists and equals $xu^{-1}$ and $(x/u)^\star = (xu^{-1})^\star = x^\star u^{-1} = x^\star/u$ for all $u \in \Inv(M)$, it straightforward to check that $x/u_1 \wedge x/u_2 \wedge \cdots \wedge x/u_n$ exists and equals $x/t$.  Therefore we have
$$(x/t)^\star = (x/u_1)^\star \wedge \cdots \wedge (x/u_n)^\star = x^\star/u_1 \wedge \cdots \wedge x^\star/u_n
= x^\star/t.$$
The proof for $t \backslash x$ is similar, so $\star$ is stable.  Also, if $M$ is a meet semilattice and $x,t \in M$ with $t$ compact, then, writing $t = u_1 \vee u_2 \vee \cdots \vee u_n$ with each $u_i \in \Inv(M)$, one easily verifies that $x/t = x/u_1 \wedge x/u_2 \wedge \cdots \wedge x/u_n$ exists, and likewise for $t \backslash x$.
\end{proof}

If $M$ is an ordered unital magma and $\star \in \N(M)$, then, borrowing terminology from the theory of semistar operations, we will say that $z \in M$ is {\it $\star$-Glaz-Vasconcelos}, or {\it $\star$-GV}, if $z \leq 1$ and $z^\star = 1^\star$.  We let $\SGV(M)$ denote the set of all $\star$-GV elements of $M$, which is a submagma of $M$.  If $X \subseteq \SGV(M)$ is nonempty, then $\bigvee X \in \SGV(M)$ if $\bigvee X$ exists, and $\bigwedge X \in \SGV(M)$ if $\bigwedge X$ exists and $X$ is finite.  If $\star$ is a nucleus on a semimultiplicative lattice $Q$, then we let
$$x^{\overline{\star}} = \bigvee\{x/z: \ z \in \SGV(Q)\}$$
for all $x \in Q$, which is well-defined since $z$ is residuated for all $z \in \SGV(Q)$.  (Indeed, $z x \leq x$, and $zy \leq x$ implies $y \leq 1^*y = z^* y \leq x^*$, so $\{z \in M: \ zy \leq x\}$ is nonempty and bounded.)

\begin{lemma}\label{stableprop}
Let $\star$ be a nucleus on a semimultiplicative lattice $Q$.
\begin{enumerate}
\item $\overline{\star}$ is a preclosure on $Q$ that is smaller than $\star$, and one has
$xy^{\overline{\star}} \leq (xy)^{\overline{\star}}$ and $x^{\overline{\star}}y \leq (xy)^{\overline{\star}}$ for all $x,y \in Q$.
\item For any $x,t \in Q$ with $t$ compact one has $t \leq x^{\overline{\star}}$ if and only if $tz \leq x$ for some $z \in \SGV(Q)$.
\item If $Q$ is algebraic then $(\bigwedge X)^{\overline{\star}} = \bigwedge(X^{\overline{\star}})$ for any finite subset $X$ of $Q$ such that $\bigwedge X$ exists.
\item If $Q$ is precoherent then $\overline{\star}$ is a stable nucleus on $Q$.
\end{enumerate}
\end{lemma}

\begin{proof} \
\begin{enumerate}
\item One has $x^{\overline{\star}} \leq \bigvee\{x^\star/z: \ z \in \SGV(Q)\} =  \bigvee\{x^\star/z^\star: \ z \in \SGV(Q)\} = x^\star$ for all $x \in Q$, whence $\overline{\star}$ is smaller than $\star$, and the rest of statement (1) is equally trivial to verify.
\item Let $x,t \in Q$ with $t$ compact.  If $tz \leq x$ for some $z \in \SGV(Q)$, then $t \leq x/z \leq x^{\overline{\star}}$.  Conversely, if $t \leq x^{\overline{\star}}$, then we have $t \leq x/z_1 \vee x/z_2 \vee \cdots \vee x/z_n \leq x/(z_1 z_2 \cdots z_n)$ for some $z_1, z_2, \ldots, z_n \in \SGV(Q)$, whence $tz \leq x$, where $z = z_1 z_2 \cdots z_n \in \SGV(Q)$.
\item Let $X = \{x_1, x_2, \ldots, x_n\}$ be a finite subset of $Q$ such that $a = \bigwedge X$ exists.   We must show that $a^{\overline{\star}} = \bigwedge(X^{\overline{\star}})$.   Clearly $a^{\overline{\star}} \leq x^{\overline{\star}}$ for all $x \in X$.  Let $b$ be any lower bound of $X^{\overline{\star}}$.  If $t$ is any compact element of $Q$ such that $t \leq b$, then, by statement (2), for each $i$ there exists $z_i \in \SGV(Q)$ such that $tz_i \leq x_i$, whence $tz \leq \bigwedge X = a$, where $z = z_1 z_2 \cdots z_n \in \SGV(Q)$.  Thus $t \leq a^{\overline{\star}}$.  Taking the supremum over all such $t$ we see that $b = \bigvee\{t \in \K(Q): \ t \leq b\} \leq a^{\overline{\star}}$.  Therefore $\bigwedge(X^{\overline{\star}})$ exists and equals $a^{\overline{\star}}$.
\item We first show that $\overline{\star}$ is idempotent and therefore a nucleus on $Q$.  Let $x \in Q$.  Let $t$ be any compact element such that $t \leq (x^{\overline{\star}})^{\overline{\star}}$.  Then
$tz \leq x^{\overline{\star}}$ for some $z \in \SGV(Q)$.   If $u$ is a compact element of $Q$ such that $u \leq z$, then $tu \leq x^{\overline{\star}}$ and $tu$ is compact, whence $tu z_u \leq x$ for some $z_u \in \SGV(Q)$.  Let $z' = \bigvee\{uz_u: \ u \in \K(Q) \mbox{ and } u \leq z\}$.  Then $z' \leq 1$ and $(z')^\star  = (\bigvee\{u z_u^\star: \ u \in \K(Q) \mbox{ and } u \leq z\})^\star = z^\star = 1^\star$, whence $z' \in \SGV(Q)$.  Moreover, one has
$tz' = \bigvee\{tuz_u: \ u \in \K(Q) \mbox{ and } u \leq z\} \leq x$.  It follows that
$t \leq x^{\overline{\star}}$.  Taking the supremum over all $t$, we see that
$(x^{\overline{\star}})^{\overline{\star}} \leq x^{\overline{\star}}$.  Thus
$\overline{\star}$ is a nucleus on $Q$.  To show that $\overline{\star}$ is stable, let $x,t$ be elements of $Q$ with $t$ compact such that $x/t$ exists.  Clearly $x^{\overline{\star}}/t$ also exists and $(x/t)^{\overline{\star}} \leq x^{\overline{\star}}/t$.  Let $u \in \K(Q)$ with $u \leq x^{\overline{\star}}/t$.  Then $tu \leq x^{\overline{\star}}$ and $tu$ is compact, whence $tuz \leq x$ for some $z \in \SGV(Q)$.  Therefore $uz \leq x/t$, whence  $u \leq (x/t)^{\overline{\star}}$.  Taking the supremum over all $u$ we see that $x^{\overline{\star}}/t \leq (x/t)^{\overline{\star}}$, whence equality holds.  Combining this with statement (3), we see that $\overline{\star}$ is stable.
\end{enumerate}
\end{proof}

\begin{theorem}\label{stabletheorem}
Let $Q$ be a precoherent semimultiplicative lattice such that every compact element of $Q$ is residuated and $x \wedge 1$ exists for all $x \in Q$, and let $\star$ be a nucleus on $Q$.  For any $x \in Q$ let $x^{\overline{\star}} = \bigvee\{x/z: \ z \in \SGV(Q)\}$.
Then $\overline{\star}$ is the largest stable nucleus on $Q$ that is smaller than $\star$.  Moreover, the following conditions on $\star$ are equivalent.
\begin{enumerate}
\item $\star$ is stable.
\item $(x\wedge 1)^\star = x^\star \wedge 1^\star$ and $(x/t)^\star = x^\star/t$ for all $x,t \in Q$ with $t$ compact.
\item $(x/t \wedge 1)^\star = x^\star/t \wedge 1^\star$ for all $x,t \in Q$ with $t$ compact.
\item $\star = \overline{\star}$.
\end{enumerate}
\end{theorem}

\begin{proof}
We first show that the four statements of the proposition are equivalent.  Clearly we have $(1) \Rightarrow (2) \Rightarrow (3)$, and by Lemma \ref{stableprop}(4) we have $(4) \Rightarrow (1)$.   To show that $(3)$ implies $(4)$, we suppose that $(3)$ holds, and then we need only show that $x^\star \leq x^{\overline{\star}}$ for any $x \in Q$.  Let $t$ be any compact element of $Q$ with $t \leq x^\star$.  By the hypothesis on $Q$ the element $z = x/t \wedge 1$ exists in $Q$.  Condition $(3)$ then implies that $z^\star = x^\star/t \wedge  1^\star$.  Note then that, since $1^\star t \leq 1^\star x^\star = x^\star$ one has $x^\star/t \geq 1^\star$, whence $z^\star = x^\star/t \wedge 1^\star = 1^\star$.  Since also $z \leq 1$ we have $z \in \SGV(Q)$.  Therefore, since $zt \leq (x/t)t \leq x$, we have $t \leq x^{\overline{\star}}$.  Since this holds for all compact $t \leq x^\star$, we have $x^\star \leq x^{\overline{\star}}$, as desired.  It remains only to show that $\overline{\star}$ is the largest stable nucleus on $Q$ that is smaller than $\star$.  But by Lemma \ref{stableprop} the nucleus $\overline{\star}$ is itself a stable nucleus that is smaller than $\star$, and if $\star'$ is any other such nucleus, then one has
$\star' = \overline{\star'} \leq \overline{\star}$, whence $\overline{\star}$ is larger than $\star'$.
\end{proof}

Any precoherent multiplicative lattice satisfies the hypotheses of Theorem \ref{stabletheorem} above.  By Propositions \ref{ucoherent} and \ref{stableU} and Corollary \ref{ulattices}, so does any commutative associative unital semi-$\U$-lattice $Q$ with $1$ compact that is also a meet semilattice, such as the near $\U$-lattice $\F(D)$ and the semi-$\U$-lattice $\Fp(D)$ for any integral domain $D$.

Let $\star_w = \overline{\star_f}$ for any nucleus $\star$ on a precoherent semimultiplicative lattice $Q$.

\begin{proposition}\label{stablecor}
Let $\star$ be a nucleus on a coherent semimultiplicative lattice $Q$.
\begin{enumerate}
\item One has $x^{\star_w} = \bigvee\{x/z: \ z \in \SGV(Q)\cap \K(Q)\}$ for all $x \in Q$, and $\star_w$ is finitary.  In particular, if $\star$ is finitary, then so is
$\overline{\star} = \star_w$.
\item If every compact element of $Q$ is residuated and $x \wedge 1$ exists for all $x \in Q$, then $\star_w$ is the largest stable finitary nucleus on $Q$ that is smaller than $\star$.  
\end{enumerate}
\end{proposition}

\begin{proof}
Statement (2) follows from statement (1) and Theorems \ref{stabletheorem} and \ref{klattice}.  To prove (1), let $x \in Q$.  One has $x^{\star_w} = \bigvee\{x/y: \ y\leq 1 \mbox{ and } y^{\star_f} = 1^{\star_f}\}$.  Let $t$ be any compact element of $Q$ with $t \leq x^{\star_w}$.  Then $t \leq x/y$ for some $y \in Q$ with $y \leq 1$ and $y^{\star_f} = 1^{\star_f}$.   Since $1$ is compact, the condition $y^{\star_f} = 1^{\star_f}$ implies $z^\star = 1^\star$ for some compact $z \leq y$.  Note then that $t \leq x/z$ since $tz \leq ty \leq x$.  Therefore, since $z \in \SGV(Q) \cap \K(Q)$, one
has $t \leq \bigvee\{x/z: \ z \in \SGV(Q)\cap \K(Q)\}$.  Since this holds for all compact $t \leq x^{\star_w}$, it follows that $x^{\star_w} \leq \bigvee\{x/z: \ z \in \SGV(Q)\cap \K(Q)\}$, and therefore equality holds since the reverse inequality is obvious.

Suppose now that $\star = \star_f$ is finitary.  To show that $\overline{\star}$ is also finitary,  we let $x \in Q$ and verify that $x^{\overline{\star}} = \bigvee\{y^{\overline{\star}}: \ y \in \K(Q) \mbox{ and } y \leq x\}$.  Let $t$ be any compact element of $Q$ with $t \leq x^{\overline{\star}}$.   Since $\star_w = \overline{\star}$, one has $x^{\overline{\star}} = \bigvee\{x/z: \ z \in \SGV(Q)\cap \K(Q)\}$.  Since $t \leq x^{\overline{\star}}$ is compact and the set $\{z/x:  \ z \in \SGV(Q)\cap \K(Q)\}$ is directed, it follows that $t \leq x/z$ for some $z \in \SGV(Q) \cap \K(Q)$.  Therefore $tz \leq x = \bigvee\{y \in \K(Q): \ y \leq x\}$,
whence $tz \leq y$ for some compact $y \leq x$ since $tz$ is compact.  Thus we have $t \leq y/z$, where $z \in \SGV(Q)$, whence $t \leq y^{\overline{\star}}$.  It follows that
$x^{\overline{\star}} \leq  \bigvee\{y^{\overline{\star}}: \ y \in \K(Q) \mbox{ and } y \leq x\}$, and the desired equality follows.
\end{proof}

The following result follows from Propositions \ref{vclosureprop} and \ref{stablecor} and Theorems \ref{klattice} and \ref{stabletheorem}.

\begin{proposition}
Let $Q$ be a precoherent near multiplicative lattice in which  every compact element of $Q$ is residuated.  Suppose moreover that $x \wedge 1$ exists for all $x \in Q$.  Let $\star$ be a nucleus on $Q$.
\begin{enumerate}
\item The nucleus $\overline{v}(a)= \overline{v(a)}$ exists for any $a \in Q$ and is the largest stable nucleus on $Q$ such that $a^{\overline{v}(a)} = a$, and $\overline{\star} = \bigwedge\{\overline{v}(a): \ a\in Q^{\overline{\star}}\}.$
\item If $1 \in Q$ is compact, then the nucleus $w(a) = \overline{t(a)}$ exists for any $a \in Q$ and is the largest stable finitary nucleus on $Q$ such that $a^{w(a)} = a$, and $\star_w = \bigwedge\{w(a): \ a \in Q^{\star_w}\}.$
\end{enumerate}
\end{proposition}

\begin{lemma}
Let $Q$ be a near prequantale.  Then $\bigwedge \Gamma$ is a stable nucleus on $Q$ for any set $\Gamma$ of stable nuclei on $Q$.  
\end{lemma}

\begin{proof}
Let $X$ be a finite subset of $Q$ that is bounded below.  Then one has
$\left(\bigwedge X\right)^{\bigwedge \Gamma} = \bigwedge\left\{\left(\bigwedge X\right)^\star: \ \star \in \Gamma\right\} = \bigwedge\left\{\bigwedge (X^\star): \ \star \in \Gamma\right\} = \bigwedge\{x^\star: \ x \in X, \ \star \in \Gamma\} = \bigwedge \left(X^{\bigwedge \Gamma}\right)$.
Moreover, for any $x, t \in Q$ with $t$ compact, if $x/t$ exists, then one has
$(x/t)^{\wedge \Gamma} =  \bigwedge\{(x/t)^\star: \ \star \in \Gamma\} =  \bigwedge\{x^\star/t: \ \star \in \Gamma\} = x^{\bigwedge \Gamma}/t$,
and a similar proof holds for $t \backslash x$ if $t \backslash x$ exists.  Thus
$\bigwedge \Gamma$ is stable.
\end{proof}

\begin{corollary}
Let $Q$ be a precoherent near multiplicative lattice in which  every compact element of $Q$ is residuated.  Suppose moreover that $x \wedge 1$ exists for all $x \in Q$.  A nucleus $\star$ on $Q$ is stable if and only if $\star = \bigwedge \{\overline{v}(a): \ a \in X\}$ for some subset $X$ of $Q$.
\end{corollary}



\section{Star and semistar operations}\label{sec:ASO}

Throughout this section let $D$ denote an integral domain with quotient field $F$.  A {\it Kaplansky fractional ideal} of $D$ is a $D$-submodule of $F$.  Let $\F(D)$ denote the ordered monoid of all nonzero Kaplansky fractional ideals of $D$.  Recall that a {\it semistar operation} on $D$ is a closure operation $\star$ on the poset $\F(D)$ such that $(aI)^\star = aI^\star$ for all nonzero $a \in F$ and all $I \in \F(D)$.  The submonoid $\P(D)$ of $\F(D)$ of all nonzero principal $D$-submodules of $F$ is a sup-spanning subset of $\F(D)$ contained in $\Inv(\F(D))$.  Therefore, by Propositions \ref{closureprop2} and \ref{closureprop3}, a semistar operation on $D$ is equivalently a nucleus on the ordered monoid $\F(D)$.  Together with Propositions \ref{closureprop1} and \ref{closureprop1a}, this yields a proof of Theorem \ref{mainprop1} of the introduction.

A semistar operation on $D$ may be defined alternatively as a closure operation on the poset $\operatorname{Mod}_D(F)$ such that $(aI)^\star = aI^\star$ for all $a \in F$ and all $I \in \operatorname{Mod}_D(F)$, or equivalently a nucleus $\star$ on the $\K$-lattice $\operatorname{Mod}_D(F)$ such that $(0)^\star = (0)$.  This alternative definition is advantageous because $\F(D)$ is typically neither complete nor residuated, while every $\K$-lattice is complete and residuated, and the definition thereby eliminates the need for many results, such as \cite[Proposition 5, Theorem 20, Lemma 40]{oka}, to make exception for the zero ideal.  The definition also allows for the possibility of relaxing the condition $(0)^\star = (0)$ by considering all nuclei on $\operatorname{Mod}_D(F)$.  Since $\F(D)$ inherits its ``near'' properties (such as being a near $\K$-lattice and near $\U$-lattice) from corresponding ``complete'' properties of $\operatorname{Mod}_D(F)$, all of the results of this paper generalizing results on semistar operations apply as well to all nuclei on $\operatorname{Mod}_D(F)$.

Recall that a $D$-submodule $I$ of $F$ is said to be a {\it fractional ideal} of $D$ if $aI \subseteq D$ for some nonzero element $a$ of $F$.   Although the ordered monoid $\Fp(D)$  of all nonzero fractional ideals of $D$ is not necessarily a near prequantale, it is a semiprequantale (and in fact a semi-$\U$-lattice).  Moreover, $\Fp(D)$ is coherent and residuated.  A {\it star operation} $*$ on $D$ is a closure operation on the poset $\Fp(D)$ such that $D^* = D$ and $(aI)^* = aI^*$ for all nonzero $a \in F$ and all $I \in \Fp(D)$.  The same argument as in the proof of Theorem \ref{mainprop1} yields the following.

\begin{theorem}\label{mainprop2}
Let $D$ be an integral domain with quotient field $F$.  The following conditions are equivalent for any self-map $*$ of $\Fp(D)$ such that $D^* = D$.
\begin{enumerate}
\item $*$ is a star operation on $D$.
\item $*$ is a nucleus on the ordered monoid $\Fp(D)$.
\item $*$ is a closure operation on the poset $\Fp(D)$ and $*$-multiplication on $\Fp(D)$ is associative.
\item $*$ is a closure operation on the poset $\Fp(D)$ and $(I^* J^*)^* = (IJ)^*$ for all $I,J \in \Fp(D)$ (or equivalently such that the map $*: \Fp(D) \longrightarrow \Fp(D)^*$ is a magma homomorphism).
\item $HJ \subseteq I^*$ if and only if $H J^* \subseteq I^*$ for all $H,I,J \in \Fp(D)$.
\item $(I^* : J) = (I^* : J^*)$ for all $I, J \in \Fp(D)$.
\item $* = \star|_{\Fp(D)}$ for some semistar operation $\star$ on $D$ such that $D^\star = D$.
\end{enumerate}
\end{theorem}

The following proposition, which follows from the results of Sections \ref{sec:COCL} and \ref{sec:PFN}, lists properties of the posets $\SStar(D)$, $\SStar_f(D)$, $\Star(D)$, and $\Star_f(D)$ of all semistar operations, finite type semistar operations, star operations, and finite type star operations, respectively, on $D$.  

\begin{proposition}
Let $D$ be an integral domain.  One has $\SStar(D) = \N(\F(D))$, $\SStar_f(D) = \N_f(\F(D))$, $\Star(D) = \N(\Fp(D))_{\leq v}$, and $\Star_f(D) = \N_f(\Fp(D))_{\leq t}$, and all four of these posets are complete.  Let $\Gamma \subseteq \SStar(D)$ and $\Delta \subseteq \Star(D)$, and let $\langle \Gamma \rangle$ (resp., $\langle \Delta \rangle$) denote the submonoid generated by $\Gamma$ (resp., $\Delta$) of the monoid of all self-maps of $\F(D)$ (resp., $\Fp(D)$).
\begin{enumerate}
\item  One has $$I^{\bigwedge \Gamma} = \bigcap\{I^\star: \ \star \in \Gamma\},$$ $$I^{\bigvee \Gamma} = \bigcap\{J \in \F(D): \ J \supseteq I \mbox{ and } \forall \star \in \Gamma \ (J^\star = J)\},$$ for all $I \in \F(D)$.
\item If $\Gamma \subseteq \SStar_f(D)$, then $$I^{\bigvee \Gamma} = \bigcup\{I^\gamma: \ \gamma \in \langle \Gamma \rangle\}$$ for all $I \in \F(D)$, one has $\bigvee_{\SStar_f(D)} \Gamma = \bigvee \Gamma$ and $\bigwedge_{\SStar_f(D)} \Gamma = (\bigwedge \Gamma)_f$, and if $\Gamma$ is finite then $\bigwedge_{\SStar_f(D)} \Gamma = \bigwedge \Gamma$.
\item One has $$I^{\bigwedge \Delta} = \bigcap\{I^*: \ * \in \Delta\},$$ $$I^{\bigvee \Delta} = \bigcap\{J \in \Fp(D): \ J \supseteq I \mbox{ and } \forall * \in \Delta \ (J^* = J)\},$$ for all $I \in \Fp(D)$.
\item If $\Delta \subseteq \Star_f(D)$, then $$I^{\bigvee \Delta} = \bigcup\{I^\gamma: \ \gamma \in \langle \Delta \rangle\}$$  for all $I \in \Fp(D)$, one has $\bigvee_{\Star_f(D)} \Delta = \bigvee \Delta$ and $\bigwedge_{\Star_f(D)} \Delta =  (\bigwedge \Delta)_f$, and if $\Delta$ is finite then $\bigwedge_{\Star_f(D)} \Delta = \bigwedge \Delta$.
\end{enumerate}
\end{proposition}

Next, let $J \in \F(D)$.  For all $I \in \F(D)$ we let $I^{v(J)} = (J:_F (J:_F I))$.   The operation $v(J)$ on $\F(D)$ is called {\it divisorial closure on $D$ with respect to $J$} \cite[Example 1.8(a)]{pic}.   More generally, for all $\S \subseteq \F(D)$ we define $v(\S) = \bigwedge \{v(J): J \in \S\}$, which we call {\it divisorial closure on $D$ with respect to $\S$}.
We set $t(\S) = v(\S)_f$, $\overline{v}(\S) = \overline{v(\S)}$, and $w(\S) = v(\S)_w = \overline{t(\S)}$, where for any semistar operation $\star$ one defines
$\overline{\star}$ by $I^{\overline{\star}} = \bigcup\{(I:_F J):\ J \subseteq D, \ J^\star = D^\star\}$ for all $I \in \F(D)$ and $\star_w = \overline{\star_f}$.  
A semistar operation $\star$ on $D$ is said to be {\it stable} if $(I \cap J)^\star = I^\star \cap J^\star$ for all $I,J \in\F(D)$. 
The results of Sections \ref{sec:PFN}, \ref{sec:DC},  and \ref{sec:stable} yield the following.

\begin{proposition}\label{semidivprop} 
Let $D$ be an integral domain and $J \in \F(D)$.
\begin{enumerate}
\item $v(J)$ is the largest semistar operation $\star$ on $D$ such that $J^{\star} = J$.
\item $t(J)$ is the largest finite type semistar operation $\star$  on $D$ such that $J^{\star} = J$.
\item $\overline{v}(J)$ is the largest stable semistar operation $\star$ on $D$  such that $J^{\star} = J$.
\item $w(J)$ is the largest stable finite type semistar operation $\star$ on $D$ such that $J^{\star} = J$.
\end{enumerate}
More generally, we have the following for any subset $\S$ of $\F(D)$.
\begin{enumerate}
\item[(5)] $v(\S)$ is the largest semistar operation $\star$ on $D$ such that $\S \subseteq \F(D)^\star$.
\item[(6)] $t(\S)$ is the largest finite type semistar operation $\star$ on $D$ such that $\S \subseteq \F(D)^\star$.
\item[(7)] $\overline{v}(\S)$ is the largest stable semistar operation $\star$ on $D$  such that $\S \subseteq \F(D)^\star$.
\item[(8)] $w(\S)$ is the largest stable finite type semistar operation $\star$ on $D$ such that $\S \subseteq \F(D)^\star$.
\end{enumerate}
For any semistar operation $\star$ on $D$, one has the following.
\begin{enumerate}
\item[(9)] $\star = \bigwedge\{v(J): \ J \in \F(D)^\star\}$.
\item[(10)] $\star_f = \bigwedge\{t(J): \ J \in \F(D)^{\star_f}\}$.
\item[(11)] $\overline{\star} = \bigwedge\{\overline{v}(J): \ J \in \F(D)^{\overline{\star}}\}$.
\item[(12)] $\star_w = \bigwedge\{w(J): \ J \in \F(D)^{\star_w}\}$.
\end{enumerate}
Moreover, $\star$ is stable if and only if $\star = \bigwedge\{\overline{v}(J): \ J \in \S\}$ for some $\S \subseteq \F(D)$, in which case $\star = \overline{v}(\S)$.
\end{proposition}

We are unable to characterize all subsets $\S$ of $\F(D)$ such that $\bigwedge\{t(J): \ J \in \S\}$ is of finite type (although it certainly holds if $\S$ is finite, by Proposition \ref{Nf}(3)).  We leave this as an open problem.

Propositions \ref{vquantales} and \ref{semiu} yield the following.

\begin{proposition}\label{vchar}
Let $D$ be an integral domain with quotient field $F$.  For all $I, J \in \F(D)$ one has
\begin{eqnarray*}
I^{v(J)} & = & \bigcup \{I' \in \F(D):  (J :_F I) \subseteq (J :_F I')\} \\
  & = & \bigcap \{U J: U\in \F(D) \mbox{ is invertible and } I \subseteq U J\}.
\end{eqnarray*}
\end{proposition}

Theorem \ref{stabletheorem} yields the following characterizations of stable semistar operations.

\begin{proposition}\label{stablechar}
Let $D$ be an integral domain with quotient field $F$, and let $\star$ be a semistar operation on $D$.  The following conditions are equivalent.
\begin{enumerate}
\item $\star$ is stable.
\item $(I \cap D)^\star  = I^\star \cap D^\star$ and $(I :_F J)^\star = (I^\star :_F J)$ for all $I,J \in \F(D)$ with $J$ finitely generated.
\item $(I :_D J)^\star = (I^\star :_{D^\star} J)$ for all $I,J \in \F(D)$ with $J$ finitely generated.
\end{enumerate}
\end{proposition}

Since  the ordered monoid $\Fp(D)$ of nonzero fractional ideals of a domain $D$ is a  precoherent residuated semimultiplicative lattice, a meet semilattice, and an associative unital semi-$\U$-lattice, and since the ordered monoid $\mathcal{I}(R)$ of ideals of a ring $R$ is a quantale,  results analogous to Propositions  \ref{semidivprop}, \ref{vchar}, and \ref{stablechar}  hold for star operations and for semiprime operations.

Proposition \ref{semidivprop} shows that all semistar operations can be obtained as infima of generalized divisorial closure semistar operations.  This has the potential for yielding new results on the number of semistar operations on an integral domain, which is a commonly studied problem in the theory of semistar operations.   (Similar comments hold for star operations and  for semiprime operations.) For example, combining Corollary \ref{simplenearmult} with \cite[Theorem 48]{oka}, we obtain the following.

\begin{proposition}
The following are equivalent for any integral domain $D$ with quotient field $F \neq D$. 
\begin{enumerate}
\item $D$ is a DVR.
\item The only semistar operations on $D$ are $d$ and $e$.
\item $|\SStar(D)| = 2$.
\item The near multiplicative lattice $\F(D)$ is simple.
\item $v(J) = d$ for all $J \in \F(D)$ with $J \neq F$.
\item $(J :_F (J:_F I)) = I$ for all $I,J \in \F(D)$ with $I,J \neq F$.
\end{enumerate}
\end{proposition}

Using the divisorial closure operations, one can generalize the equivalence of (1) through (3) in the theorem above as follows.

\begin{theorem}[\cite{ell}]\label{pid}
Let $D$ be a Dedekind domain with quotient field $F$ and $\Max(D)$ the set of all maximal ideals of $D$.  There is a poset embedding of the lattice $\C(2^{\Max(D)})$ of all closure operations on $2^{\Max(D)}$ into the lattice $\SStar(D)$ of all semistar operations on $D$.  Explicitly, the map $\C(2^{\Max(D)})   \longrightarrow  \SStar(D)$ acting by $* \longmapsto v(\S)$, where 
$$\S = \left\{I \in \F(D): \{\ppp: I D_\ppp = F\} \mbox{ is } *\mbox{-closed and }  \{\ppp: D_\ppp \subsetneq ID_\ppp \subsetneq F\} \mbox{ is finite} \right\},$$
is a poset embedding.  Moreover, the given embedding has an order-preserving left inverse acting by $\star \longmapsto *$, where $*$ is the largest closure operation on $2^{\Max(D)}$ such that $\{\ppp \in \Max(D): ID_\ppp = F\}$ is $*$-closed for all $I \in \F(D)^\star$.  Finally, if $\Max(D)$ is finite, then the given embedding $\C(2^{\Max(D)})   \longrightarrow  \SStar(D)$ is an isomorphism.
\end{theorem}

\begin{proof}
The theorem follows from \cite[Corollary 3.5 and Lemma 4.1]{ell}.
\end{proof}

Note that, if $D$ is a Dedekind domain, then $\{\ppp: D_\ppp \subsetneq ID_\ppp \subsetneq F\}$ is finite if and only if $I$ is a fractional ideal of some overring of $D$, if and only if $I$ is a fractional ideal of $(I:_F I)$, if and only if $((I :_F I) :_F I) = (I :_F I^2)$ is nonzero.

If $D$ and $D'$ are Dedekind domains, then by \cite[Proposition 3.1(4)]{ell} the complete lattices $\F(D)$ and $\F(D')$ are isomorphic if $\Max(D)$ and $\Max(D')$ have the same cardinality.  Combining this with Theorem \ref{pid} above, Corollary \ref{latticeisom}, \cite[Table 1]{col}, and results in \cite{alek} on the number of closure operations on $2^S$ for finite sets $S$, we obtain the following.

\begin{corollary}[{\cite[Theorem 1.3]{ell}}] Let $D$ and $D'$ be Dedekind domains.
\begin{enumerate}
\item The lattices $\SStar(D)$ and $\SStar(D')$  are isomorphic if $\Max(D)$ and $\Max(D')$ have the same cardinality. 
\item  If $\Max(D)$ is infinite, then $|\SStar(D)|$ is equal to $2^{2^{|\Max(D)|}}$.   
\item  If $|\Max(D)| = n$ is finite, then $2^{{n \choose [n/2]}} \leq |\SStar(D)| \leq 2^{2^n}$.
\item $|\SStar(D)|$ is given for $n \leq 7$ as in Table \ref{tabl}.
\end{enumerate}
\end{corollary}

\begin{table}
\caption{Number of semistar operations on a Dedekind domain $D$}\label{tabl}
\centering 
\begin{tabular}{c|c}
$|\Max(D)|$ & $|\SStar(D)|$ \\ \hline
$1$ & $2$ \\ 
$2$ & $7$ \\
$3$ & $61$ \\
$4$ & $2\ 480$ \\
$5$ & $1\ 385\ 552$ \\
$6$ & $75\ 973\ 751\ 474$ \\
$7$ & $14\ 087\ 648\ 235\ 707\ 352\ 472$
\end{tabular}
\end{table}

We speculate that it is possible to carry out similar applications of the divisorial closure operations to star operations and semiprime operations.

\begin{example}
Let $D$ be a Dedekind domain with quotient field $F$.   Let $I_D = \{\ppp \in \Max(D): ID_\ppp = F\}$ for any $I \in \F(D)$.  From \cite[Proposition 3.1(3)]{ell} one can show that
$$I^{v(J)}  =  \left\{\begin{array}{ll} K & \mbox{if }  I_D \nsubseteq J_D 
\\ \bigcap_{\ppp \notin J_D} ID_\ppp &  \mbox{if }  I_D \subseteq J_D \end{array}\right.$$
for all $I,J \in \F(D)$.    In particular, one has
$$I^{v} = \left\{\begin{array}{ll} K & \mbox{if } I_D \neq \emptyset \\ I &   \mbox{if } I_D =  \emptyset,\end{array}\right.$$ and $v(J) = v$ if and only if $J_D = \emptyset$.
Suppose that $D$ has exactly two maximal ideals $\ppp_1$ and $\ppp_2$.     Define $I^{\star_i} = I D_{\ppp_i}$ for $i = 1,2$.  One has $v(J) = \star_1$ if and only if $J_D = \{\ppp_2\}$ and $v(J)  = \star_2$ if and only if $J_D = \{\ppp_1\}$.  Also, one has $v(J) = e$ if and only if $J_D = \{\ppp_1, \ppp_2\}$.  Since $\star_1 \wedge \star_2 = d$, it follows that $\SStar(D) = \{e, v, \star_1, \star_2, \star_1 \wedge v, \star_2 \wedge v, d\} \cong  \C(2^{\{1,2\}}) \cong 2^{\{1,2,3\}} - \{\{2\}\}$, with Hasse diagram given below.
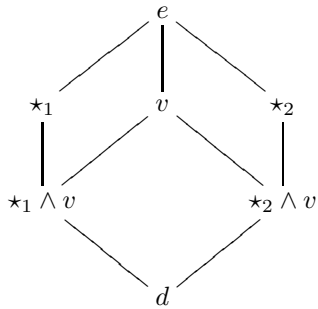
\begin{eqnarray*}
\SelectTips{cm}{11}\xymatrix{
 & e\ar@{-}[dl] \ar@{-}[d] \ar@{-}[dr] & \\
  \star_1 \ar@{-}[d]  & v \ar@{-}[dl] \ar@{-}[dr] & \star_2 \ar@{-}[d] \\
  \star_1 \wedge v \ar@{-}[dr] & & \star_2 \wedge v \ar@{-}[dl] \\
 & d & 
 }
\end{eqnarray*}
\end{example}

For any star operation $*$ on $D$, there is a unique largest semistar operation $l(*)$ on $D$ such that $I^{l(*)} = I^*$ for all $I \in \Fp(D)$.  Indeed, one may extend $*$ to $\F(D)$ by defining $I^{l(*)} = F$ for all $I \in \F(D)\setmin\Fp(D)$.   (See Proposition \ref{extendingclosures}(3).)  Using Proposition \ref{extendingclosures}(2), we may construct the unique smallest semistar operation $s(*)$ on $D$ such that $I^{s(*)} = I^*$ for all $I \in \Fp(D)$, as follows.

\begin{proposition}\label{semistarfromstar}
Let $D$ be an integral domain.
\begin{enumerate}
\item For any star operation $*$ on $D$, there exists a unique smallest semistar operation $s(*)$ on $D$ such that ${s(*)}|_{\Fp(D)} = *$, and for all $I \in \F(D)$ one has
$$I^{s(*)} = \bigcap \{K \in \F(D): \ K \supseteq I \mbox{ and } \forall J \in \Fp(D) \ (J \subseteq K \Rightarrow J^* \subseteq K)\}.$$
\item More generally, for any nucleus $\gamma$ on a submonoid $\mathcal{M}$ of $\F(D)$ containing $\P(D)$, there exists a unique smallest semistar operation $\star$ on $D$ such that $\star|_{\mathcal{M}} = \gamma$, and for all $I \in \F(D)$ one has
$$I^\star = \bigcap \{K \in \F(D): \ K \supseteq I \mbox{ and } \forall J \in \mathcal{M} \ (J \subseteq K \Rightarrow J^\gamma \subseteq K)\}.$$
\end{enumerate}
\end{proposition}

Although $l(*)$ has been studied predominantly in the literature, the semistar operation $s(*)$ is a useful alternative.  Note in particular that $s(d) = d$ on any domain $D$, and by Proposition \ref{sstarfintype} below $*$ is of finite type if and only if $s(*)$ is of finite type, where a star or semistar operation $\star$ is said to be {\it of finite type} if $\star = \star_f$.  By contrast, although the star operation $d$ is of finite type, the semistar operation $l(d)$ on $D$ is of finite type if and only if $l(d) = d$ on $D$, if and only if $D$ is {\it conducive}, that is, every overring of $D$ other than $F$ is a fractional ideal of $D$.  

\begin{proposition}\label{sstarfintype}
Let $D$ be an integral domain.  For any semistar operation $\star$ and star operation $*$ on $D$, one has $\star|_{\Fp(D)} = *$ if and only if $s(*) \leq \star \leq l(*)$, in which case
$s(*_f) = s(*)_f = \star_f = l(*)_f$.  In particular, $*$ is a finite type star operation if and only if $s(*)$ is a finite type semistar operation, in which case $s(*)$ is the unique finite type semistar operation $\star$ on $D$ such that $\star|_{\Fp(D)} = *$.
\end{proposition}

\begin{proof}
The first equivalence is clear from the construction of $s(*)$ and $l(*)$.  Suppose that $\star|_{\Fp(D)} = *$.   For any $I \in \F(D)$, one has $I^{\star_f} = \bigcup_{J \subseteq I \textup{ f.g.}} J^*$.  It follows that $s(*)_f = \star_f = l(*)_f$.  It also follows that if $*$ is of finite type, then $s(*)_f|_{\Fp(D)} = *$, whence $s(*) \leq s(*)_f$, whence $s(*)$ is of finite type.  In particular, one has $s(*_f) = s(*_f)_f \leq s(*)_f$.  To prove the reverse inequality, let $I \in \F(D)$, and note as above that $I^{s(*)_f} =\bigcup_{J \subseteq I \textup{ f.g.}} J^*$.  It follows that if $K \supseteq I$, and if $J \subseteq K$ implies $J^{*_f} \subseteq K$ for all $J \in \F(D)$, then $I^{s(*)_f} \subseteq K$.  Therefore $I^{s(*)_f} \subseteq I^{s(*_f)}$.  Thus we have $s(*)_f \leq s(*_f)$.  The rest of the proposition follows.
\end{proof}

The star operation $v = v(D)$ is the largest star operation on $D$, and $t = v_f$ is the largest finite type star operation on $D$.  One has $l(v) = v$ and $l(t) = t$, but one may have $s(v) < l(v)$ and $s(t) < l(t)$.  Indeed, let $D$ be any divisorial domain, that is, any domain for which $I^v = I$ for all $I \in \Fp(D)$.  Since $v = t = d$ on $D$, it follows that $s(v) = s(t) = d$ on $D$ as well.  However, if $D$ is non-conducive, say, if $D = \ZZ$, then $d < l(t) \leq l(v)$, whence $s(v) < l(v)$ and $s(t) < l(t)$.

Finally, we comment on the requirement $D^* = D$ for a star operation $*$.  Let us say that a {\it quasistar operation} is a self-map $*$ of $\Fp(D)$ that satisfies all of the star operation axioms except possibly $D^* = D$.  Equivalently, a quasistar operation is a nucleus on the semimultiplicative lattice $\Fp(D)$.  The poset $\QStar(D) = \N(\Fp(D))$ of all quasistar operations on $D$ is of course bounded complete, and $\Star(D) = \QStar(D)_{\leq v}$.  An interesting question arises as to when $\QStar(D)$ is complete, that is, when there is a largest quasistar operation.  We show that the  answer is in the affirmative if and only if $(D :_F \widetilde{D}) \neq (0)$, where $\widetilde{D}$ is the complete integral closure of $D$ \cite{gilmer} and $F$ is the quotient field of $D$. 

 An element $x$ of $F$ is said to be {\it almost integral over D} if there exists a nonzero element $c$ of $D$ such that $c x^n \in D$ for all positive integers $n$, or equivalently, if $(D:_F D[x]) \neq (0)$. The {\it complete integral closure $\widetilde{D}$ of $D$} is the set of all elements of $F$ that are almost integral over $D$.  The set $\widetilde{D}$ is an overring of $D$ containing the integral closure $\overline{D}$ of $D$, and if $D$ is Noetherian then $\widetilde{D} = \overline{D}$.  One says that $D$ is {\it completely integrally closed} if $\widetilde{D} = D$.  In general $\widetilde{D}$ may not be completely integrally closed; but in the next section we will show that there is a smallest completely integrally closed overring $D^\sharp$ of $D$, which we call the {\it complete complete integral closure of $D$}.

\begin{proposition}
Let $D$ be an integral domain with quotient field $F$.  If $*$ is any quasistar operation on $D$, then $D^* \subseteq \widetilde{D}$.  Moreover, the following conditions are equivalent.
\begin{enumerate}
\item $\QStar(D)$ is a complete lattice.
\item There is a largest quasistar operation $e$ on $D$.
\item There is a largest overring $D'$ of $D$ such that $(D :_F D') \neq (0)$.
\item $(D :_F \widetilde{D}) \neq (0)$.
\end{enumerate}
If the above conditions hold, then one has $D^e = D' = \widetilde{D} = D^\sharp$ and $e = v(D^e)$.  Moreover, the above conditions hold if $\widetilde{D}$ is finitely generated as a $D$-algebra, hence also if $D$ is completely integrally closed or if $D$ is Noetherian and $\overline{D}$ is module finite over $D$.
\end{proposition}

\begin{proof}
If $*$ is a quasistar operation on $D$, then $(D :_F D^*) \neq (0)$, which implies that $(D :_F D[x]) \neq (0)$ for all $x \in D^*$, and therefore $D^* \subseteq \widetilde{D}$.  Clearly (1) and (2) are equivalent.  If $e$ is a largest quasistar operation on $D$ and $E$ is any overring of $D$ such that $(D :_F E) \neq (0)$, then the operation $*_{E}: I \longmapsto IE$ is a quasistar operation on $D$, and therefore $*_E \leq e$ and so $E = D^{*_{E}} \subseteq D^e$.  Therefore (2) implies (3).  Suppose that (3) holds.  Then for any quasistar operation $*$ one must have  $D'^* = D'$ and therefore $* \leq v(D')$, so that $e = v(D')$ is the largest quasistar operation on $D$.  Furthermore, if $x \in \widetilde{D}$, then $(D :_F D[x]) \neq (0)$, so that $D[x] \subseteq D'$ and therefore $x \in D'$.  Conversely, one has $D' \subseteq \widetilde{D}$, so that $D' = \widetilde{D}$.  Therefore (3) implies both (2) and (4).  Since also (4) implies (3), all four conditions are equivalent.  

Suppose, now, that the four conditions hold.  We have seen already that $D^e = D' = \widetilde{D}$ and $e = v(D^e)$.  Since $(D :_F \widetilde{D}) \neq (0)$, by \cite[Corollary 6]{gilmer} the domain $\widetilde{D}$ is completely integrally closed and therefore $\widetilde{D} = D^\sharp$.  Finally, the last statement of the proposition is clear.
\end{proof}

\section{Integral, complete integral, and tight closure}\label{sec:ICCICTC}

Let us recall the definition of the integral closure semistar operation.  Let $D$ be an integral domain with quotient field $F$.  If $I \in \F(D)$, then $x \in F$ is {\it integral over I} if there exists a positive integer $n$ and elements $a_i \in I^i$ such that $x^n + a_1 x^{n-1} + \cdots + a_{n-1}x + a_n = 0$.  The {\it semistar integral closure of $I$} is the set $\overline{I}$ of all elements of $F$ that are integral over $I$ and is given by $\overline{I} = \bigcap\{IV: V \textup{ is a valuation overring of } D\}$.  It follows that semistar integral closure is a semistar operation on $D$.  In the literature of semistar operations it is often called the {\it $b$-operation}.  Semistar integral closure is of finite type.  Indeed, if $x \in \overline{I}$, say, if $x^n + a_1 x^{n-1} + \cdots + a_{n-1}x + a_n = 0$ for elements
$a_i = \sum_j b_{ij1}b_{ij2}\cdots b_{iji} \in I^i$, where each $b_{ijk}$ lies in $I$, then $x \in \overline{J}$, where $J \subseteq I$ is the ideal generated by the $b_{ijk}$.  
For any ideal $I$ of $D$, the ideal $\overline{I} \cap D$ is known as the {\it integral closure} of $I$.

The operation of complete integral closure can also be used to define a semistar operation. If $I \in \F(D)$, then $x \in F$ is {\it almost integral over I} if there exists a nonzero element $c$ of $F$ such that $c x^n \in I^n$ for all positive integers $n$.  We define $\widetilde{I}$ to be the set of all elements of $F$ that are almost integral over $I$, and we say that $I$ is {\it completely integrally closed} if $\widetilde{I} = I$.  Although one may have $\widetilde{D}\subsetneq \widetilde{\widetilde{D}}$, the map $I \longmapsto \widetilde{I}$ is at least a preclosure on $\F(D)$. We define
$$I^\sharp = \bigcap\{J \in \F(D): J \supseteq I \mbox{ and } J \mbox{ is completely integrally closed}\},$$
which we call the {\it semistar complete integral closure} of $I$.  Since $a\widetilde{J} = \widetilde{aJ}$ for all nonzero $a \in F$, and therefore $I\widetilde{J} \subseteq \widetilde{IJ}$, for all $I,J \in \F(D)$, by Lemma \ref{preclosurelemma} the operation $\sharp$ is a semistar operation on $D$.  In particular, we have the following.

\begin{proposition}
Let $D$ be an integral domain.  The operation $\sharp$ is the unique semistar operation on $D$ such that $I^\sharp = I$ for $I \in \F(D)$ if and only if $I$ is completely integrally closed.  Moreover,  for all $I \in \F(D)$ one has $\overline{I} \subseteq \widetilde{I} \subseteq \widetilde{I^\sharp} = I^\sharp$, with equalities holding if $D$ is Noetherian.  Thus $I^\sharp$ is the smallest completely integrally closed Kaplansky fractional ideal of $D$ containing $I$, and $\sharp$ equals the semistar integral closure operation on $D$ if $D$ is Noetherian. 
\end{proposition}

\begin{corollary}
For any integral domain $D$ with quotient field $F$, one has the following.
\begin{enumerate}
\item $D^\sharp$ is the smallest completely integrally closed overring of $D$. 
\item Define $D^{\sharp_0} = D$, $D^{\sharp_{\alpha+1}} = \widetilde{D^{\sharp_{\alpha}}}$ for all successor ordinals $\alpha+1$, and $D^{\sharp_{\alpha}}  = \bigcup_{\beta < \alpha} D^{\sharp_\beta}$ for all limit ordinals $\alpha$.  Then $D^\sharp = D^{\sharp_\alpha}$ for all $\alpha  \gg 0$.
\item If $J \in \F(D)$ is completely integrally closed, then the largest overring $(J:_F J)$ of $D$ for which $J$ is a Kaplansky fractional ideal is also completely integrally closed.
\end{enumerate}
\end{corollary}

Next, let $D^+$ denote the integral closure of $D$ in an algebraic closure of $F$.  For all $I \in \F(D)$, let $I^\pi = ID^+ \cap F$, which we call the {\it semistar plus closure} of $I$, and which is easily seen to define a semistar operation $\pi$ on $D$.   Note that $D^\pi = \overline{D}$ and therefore $I\overline{D} \subseteq I^\pi$ for all $I \in \F(D)$.  Like semistar integral closure, semistar plus closure is of finite type.  Also, for any ideal $I$ of $D$, the ideal $I^\pi \cap D = ID^+ \cap D$ is known as the {\it plus closure} of $I$.

The operation of tight closure also yields a semistar operation, as follows.  Let $D$ be any integral domain (not necessarily Noetherian) of characteristic $p > 0$ with quotient field $F$.  For any $I \in \F(D)$ let  $$I^T = \{x \in F: \ \exists c \in F\setmin\{0\} \mbox{ such that } cx^q \in I^{[q]} \mbox{ for all  } q = p^e, e \geq 0\},$$
where $I^{[q]}$ for any $q = p^e$ denotes the $D$-submodule of $F$ generated by the image of $I$ under the $q$-powering Frobenius endomorphism of $F$.  It is clear that $T$ is a preclosure on $\F(D)$.   Moreover, one has $aJ^T = (aJ)^T$ for all nonzero $a \in F$, and therefore $IJ^T \subseteq (IJ)^T$, for all $I,J \in \F(D)$.  Therefore by Lemma \ref{preclosurelemma} the operation $\tau$ on $\F(D)$ of {\it semistar tight closure}, defined by $$I^\tau = \bigcap\{J \in \F(D): \ J \supseteq I \mbox{ and } J^T = J\}$$
for all $I \in \F(D)$, is a semistar operation on $D$.  Note that by \cite[Corollary 4]{AZ} one has $\widetilde{D} \subseteq D^T \subseteq \widetilde{\widetilde{D}}$,
from which it follows that $D^\tau = D^\sharp$.  In particular, $\tau$ restricts to a star operation on $D$ if and only if $D^\sharp = D$ if and only if $D$ is completely integrally closed.  Also, if $D$ is Noetherian then $I^\tau = I^T \subseteq \overline{I}$ for all $I \in \F(D)$.

\begin{example}\label{tightexample}
Let $k$ be a finite field of characteristic $p > 0$, and let $D$ be the subring $k[X,XY,XY^p, XY^{p^2}, \ldots]$ of $k[X,Y]$.  Then $D^\tau = D^T = \widetilde{\widetilde{D}} = k[X,Y]$, but $Y \notin \widetilde{D}$, whence $\widetilde{D} \subsetneq D^\tau$.
\end{example}

The {\it tight closure} $I^*$ of an ideal $I$ of $D$ may be defined as the ideal $I^* = I^\tau \cap D$.  Since $\tau = T$ if $D$ is Noetherian, this definition coincides with the well-known definition of tight closure in the Noetherian case.  In fact, we may generalize this definition to an arbitrary commutative ring $R$ of prime characteristic $p$, as follows.  For any $I \in \mathcal{I}(R)$ we let $$I^T = \{x \in R: \ \exists c \in R^o \mbox{ such that } cx^q \in I^{[q]} \mbox{ for all  } q = p^e \gg 0\},$$ where $R^o$ is the complement of the union of the minimal primes of $R$ and $I^{[q]}$ is the ideal of $R$ generated by the image of $I$ under the $q$-powering Frobenius endomorphism of $R$.  We say that $I \in \mathcal{I}(R)$ is {\it tightly closed} if $I^T = I$.  It is clear that $T$ is a preclosure on the complete lattice $\mathcal{I}(R)$.  For any $I \in \mathcal{I}(R)$ we let
$$I^* = \bigcap\{J \in \mathcal{I}(R): \ J \supseteq I \mbox{ and } J \mbox{ is tightly closed}\}.$$
Since $T$ is a preclosure on $\mathcal{I}(R)$ such that $IJ^T \subseteq (IJ)^T$ for all $I,J \in \mathcal{I}(R)$, we obtain from Lemma \ref{preclosurelemma} the following result.

\begin{proposition}
Let $R$ be a commutative ring of prime characteristic $p$.  The operation $*$ is the unique semiprime operation on $R$ such that $I^* = I$ for $I \in \mathcal{I}(R)$ if and only if $I$ is tightly closed.  Moreover,  for all $I \in \mathcal{I}(R)$ one has $I^T \subseteq (I^*)^T = I^*$, with equalities holding if $R$ is Noetherian.  In particular, $I^*$ for any $I \in \mathcal{I}(R)$ is the smallest tightly closed ideal of $R$ containing $I$, and $*$ equals the ordinary tight closure operation if $R$ is Noetherian.
\end{proposition}

Thus $*$ provides a potential definition of tight closure for non-Noetherian commutative rings of prime characteristic.  The finitary nucleus $*_f$ on $\mathcal{I}(R)$ is another candidate for a definition of tight closure in the non-Noetherian case.  Further investigation may determine whether or not either of these generalizations of tight closure is useful.  One may seek, for example, a generalization of Tight Closure via Colon-Capturing \cite[Theorem 3.1]{hun} with respect to some definition of non-Noetherian Cohen-Macaulay rings, such as the notion presented in \cite{ham}.

\section{Ideal and module systems}\label{sec:AMSIS}

For any magma $M$, let $M_0$ denote the magma $M \amalg \{0\}$, where $0x = 0 = x0$ for all $x \in M_0$.  A {\it module system} on an abelian group $G$ is a closure operation $r$ on the $\K$-lattice $2^{G_0}$ such that $\emptyset^r = \{0\}$ and $(cX)^r = cX^r$ for all $c \in G_0$ and all $X \in 2^{G_0}$.

\begin{theorem}\label{modulesystems}
A module system on an abelian group $G$ is equivalently a nucleus $r$ on the multiplicative lattice $2^{G_0}$ such that $\emptyset^r = \{0\}$.   Moreover, the following are equivalent for any self-map $r$ of $2^{G_0}$ such that $\emptyset^r = \{0\}$.
\begin{enumerate}
\item $r$ is a module system on $G$.
\item $r$ is a closure operation on the poset $2^{G_0}$ and $r$-multiplication is associative.
\item $r$ is a closure operation on the poset $2^{G_0}$ and $(X^r Y^r)^r = (XY)^r$ for all $X,Y \in 2^{G_0}$.
\item $XY \subseteq Z^r$ if and only if $XY^r \subseteq Z^r$ for all $X,Y,Z \subseteq G_0$.
\end{enumerate}
\end{theorem}

\begin{proof}
Let $\Sigma = \{\{c\}: c \in G_0\}$.  Let $r$ be a module system on $G$.  Since  $\Sigma$ is a sup-spanning subset of $2^{G_0}$ and $\Sigma \subseteq \T^r(2^{G_0})$, it follows from Corollary \ref{joinspan} that $r$ is a nucleus on $2^{G_0}$.  Conversely, let $r$ be a nucleus on $2^{G_0}$ such that $\emptyset^r = \{0\}$.  Since $\Sigma\setmin\{\{0\}\} \subseteq \Inv(2^{G_0})$, by Proposition \ref{closureprop2}, we have $(\{c\}X)^r = \{c\}X^r$ for all $c \in G$ and all $X \in 2^{G_0}$.  Also, we have $(\{0\}X)^r = \{0\}^r = \{0\} = \{0\}X^r$ for all $X \in 2^{G_0}$.  Thus $r$ is a module system on $G$.
\end{proof}

An alternative proof of Theorem \ref{mainprop1} can be given using Theorem \ref{modulesystems}, Propositions \ref{closureprop1} and \ref{closureprop1a}, and the following lemma.

\begin{lemma}\label{ssexample} Let $D$ be an integral domain, $F^\times$ the group of nonzero elements of the quotient field $F$ of $D$, and $\star$ a self-map of $\F(D)$.  Let $\emptyset^r = \{0\}^ r = \{0\}$ and $X^r = (DX)^\star$ for all subsets $X$ of $F$ containing a nonzero element, where $DX$ denotes the $D$-submodule of $F$ generated by $X$.  Then $\star$ is a semistar operation on $D$ if and only if $r$ is a module system on $F^\times$.
\end{lemma}

Next, a {\it weak ideal system} \cite{hal1} on a commutative monoid $M$ is a closure operation $r$ on the $\K$-lattice $2^{M_0}$ such that $0 \in \emptyset^r$, $cM_0 \subseteq \{c\}^r$, and $cX^r \subseteq (cX)^r$ for all $c \in M_0$ and all $X \in 2^{M_0}$.  A weak ideal system on $M$ is said to be an {\it ideal system} on $M$ if $(cX)^r = cX^r$ for all such $c$ and $X$.  We have the following result, whose proof is similar to that of Theorem \ref{modulesystems}.

\begin{proposition}\label{idealystems}
Let $M$ be a commutative monoid.  A weak ideal system on $M$ is equivalently a nucleus $r$ on the $\K$-lattice $2^{M_0}$ such that $\{0\}^r = \emptyset^r$ and $\{1\}^r = M_0$. Moreover, a weak ideal system $r$ on $M$ is an ideal system on $M$ if and only if every singleton in $2^{M_0}$ is transportable through $r$. 
\end{proposition}

A module system on an abelian group $G$ may be seen equivalently as a nucleus on the $\K$-lattice $2^{G_0} \setmin 2^G$.  Since $2^{G_0}$ and $2^{G_0} \setmin 2^G$ are (coherent) $\K$-lattices and $2^{G_0} \setmin 2^G$ is a $\U$-lattice, most of the results of this paper, and in particular the results of Sections \ref{sec:MCO} through \ref{sec:stable}, apply specifically to module systems.  Our results on star and semistar operations generalize to module systems, since the ordered monoids $2^{G_0}-2^G$ and $\F(D)$, where $D$ is any integral domain, share the relevant properties.  Finally, as $2^{M_0}$ is a $\K$-lattice for any commutative monoid $M$, many of our results apply as well to weak ideal systems.  We leave it to the interested reader to carry out the details.



\end{document}